\documentclass[leqno,final]{siamltex}
\usepackage{graphicx} 
\usepackage{amsmath,amstext,amssymb,bm}
\numberwithin{equation}{section}
\usepackage{color}
\usepackage{tikz}
\usetikzlibrary{shapes,arrows}

\newtheorem{remark}{Remark}[section]
\newtheorem{thm}{Theorem}[section]

\newtheorem{lem}{{Lemma}}[section]

\allowdisplaybreaks[4]

\newcommand{\eps}{\varepsilon}
\newcommand{\Ome}{{\Omega}}

\newcommand{\p}{{\partial}}  
\newcommand{\Del}{{\Delta}}
\newcommand{\nab}{\nabla}
\newcommand{\nonum}{\nonumber}
\renewcommand{\(}{\bigl(} 
\renewcommand{\)}{\bigr)}
\newcommand{\bl}{\bigl\langle}
\newcommand{\br}{\bigr\rangle}

\newcommand{\mce}{\mathcal{E}_h}
\newcommand{\mcei}{\mce^I}
\newcommand{\mceb}{\mce^B}
\newcommand{\mct}{\mathcal{T}_h}
\newcommand{\avg}[1]{\bigl\{\hspace{-0.1cm}\bigl\{#1\bigr\}\hspace{-0.1cm}\bigr\}}
\newcommand{\jump}[1]{\bigl[\hspace{-0.075cm}\bigl[#1\bigr]\hspace{-0.075cm}\bigr]}
\newcommand{\vh}{V_h}
\newcommand{\les}{\lesssim}

\newcommand{\Lcal}{\mathcal{L}}
\newcommand{\Lcalh}{\mathcal{L}_h}

\newcommand{\Lcalo}{\mathcal{L}_{\text{\rm  0}}}
\newcommand{\Lcaloh}{\mathcal{L}_{\text{\rm 0},h}}

\newcommand{\An}{A_{\text{\rm 0}}}

\newcommand{\xn}{x_{\text{\rm 0}}}

\newcommand{\bw}{\bm w}
\newcommand{\bn}{\bm n}

\newcommand{\cTho}{\mathcal{T}_h(\Ome)}

\newcommand{\cQ}{\mathcal{Q}_h}

\newcommand{\ind}{\chi_{B_{R^{\prime\prime}}}}

\begin{document}

\title{$C^0$ discontinuous Galerkin finite element methods for second order linear elliptic partial 
differential equations in non-divergence form\thanks{This work was partial supported by the NSF through 
grants DMS-1016173 and DMS-1318486 (Feng),  and DMS-1417980 (Neilan),
and the Alfred Sloan Foundation (Neilan).}}

\author{Xiaobing Feng\thanks{Department of Mathematics, The University of Tennessee, 
Knoxville, TN 37996 (xfeng@math.utk.edu).}
\and{Lauren Hennings}\thanks{Department of Mathematics, University of Pittsburgh, 
Pittsburgh, PA 15260 (LNH31@pitt.edu).}
\and{Michael Neilan}\thanks{Department of Mathematics, University of Pittsburgh, 
Pittsburgh, PA 15260 (neilan@pitt.edu).} }

\date{}

\maketitle

\thispagestyle{empty}

\begin{abstract}
This paper is concerned with finite element approximations of $W^{2,p}$ strong solutions of 
second-order linear elliptic partial differential equations (PDEs) in non-divergence form with 
continuous coefficients. A nonstandard (primal) finite element method, which uses finite-dimensional 
subspaces consisting globally continuous piecewise polynomial functions, is proposed and analyzed.
The main novelty of the finite element method is to introduce an interior penalty term, which 
penalizes the jump of the flux across the interior element edges/faces, to augment a 
nonsymmetric piecewise defined and PDE--induced bilinear form. Existence, uniqueness 
and error estimate in a discrete $W^{2,p}$ energy norm are proved for the proposed finite element 
method. 
This is achieved by 
establishing a discrete Calderon--Zygmund--type estimate and 
mimicking strong solution PDE techniques at the discrete level. Numerical experiments 
are  provided to test the performance of proposed finite element method and to
validate the convergence theory.
\end{abstract}

\section{Introduction}\label{sec-1}
In this paper we consider finite element approximations of the following linear elliptic PDE
in non-divergence form:
\begin{subequations} \label{problem}
\begin{alignat}{2} \label{problem1}
\Lcal u:=-A:D^2 u & = f\qquad &&\text{in }\Ome,\\
u& = 0\qquad &&\text{on }\p\Ome.
\end{alignat}
\end{subequations}
Here, $\Ome\subset \mathbb{R}^n$ is an open bounded domain with
boundary $\p\Ome$, $f\in L^p(\Omega)\ (1<p<\infty)$ is given,
and 
$A=A(x)\in \big[C^{0}(\overline{\Ome})\big]^{n\times n}$ is a positive
definite matrix on $\overline{\Omega}$, but not necessarily differentiable.
Problems such as \eqref{problem} arise in fully nonlinear elliptic Hamilton-Jacobi-Bellman 
equations, a fundamental problem in the field of stochastic optimal control 
\cite{FlemingSonerBook,JensenSmears13}. In addition, elliptic PDEs in non-divergence 
form appear in the linearization and numerical methods of fully nonlinear second order 
PDEs \cite{CaffarelliGut97,FengNeilan09,Neilan13}.

Since $A$ is not smooth, the PDE \eqref{problem1} cannot be written in divergence form,
and therefore notions of weak solutions defined by variational principles are not applicable. 
Instead, the existence and uniqueness of solutions are generally sought in the classical 
or strong sense. In the former case, Schauder theory states the existence of a unique 
solution $u\in C^{2,\alpha}(\Omega)$ to \eqref{problem} provided the coefficient matrix 
and source function are H\"older continuous, and if the boundary satisfies 
$\p\Omega\in C^{2,\alpha}$.  In the latter case, the Calderon-Zygmund theory states the existence 
and uniqueness of $u\in W^{2,p}(\Omega)$ satisfying \eqref{problem} almost everywhere provided
$f\in L^p(\Omega)$, $A\in [C^0(\overline{\Omega})]^{n\times n}$ and $\p\Omega\in C^{1,1}$.
In addition, the existence of a strong solution to \eqref{problem} in two-dimensions
and on convex domains is proved in \cite{LadyBook,Bernstein1910,BabuskaOsborn94}.

Due to their non-divergence structure, designing convergent numerical methods, in particular, 
Galerkin-type methods, for problem \eqref{problem} has been proven to be difficult. Very
few such results are known in the literature. Nevertheless, even problem \eqref{problem} 
does not naturally fit within the standard Galerkin framework, several finite element methods 
have been recently proposed.  In \cite{LakkisPryer11} the authors 
considered mixed finite element methods using Lagrange finite element spaces for problem 
\eqref{problem}.  An analogous discontinuous Galerkin (DG) method 
was proposed in \cite{DednerPryer13}.  The convergence analysis of these methods
for non-smooth $A$ remains open.  A least-squares-type discontinuous Galerkin 
method for problem \eqref{problem} with coefficients satisfying 
the Cordes condition was proposed and analyzed in \cite{SmearsSuli}.
Here, the authors established optimal order estimates in $h$ with respect to a $H^2$-type norm.

The primary goal of this paper is to develop a structurally simple and computationally 
easy finite element method for problem \eqref{problem}. Our method is a primal method 
using Lagrange finite element spaces. The method is well defined  
for all polynomials degree greater than one and can be easily implemented on current 
finite element software. Moreover, our finite element method resembles interior
penalty discontinuous Galerkin (DG) methods in its formulation and its bilinear form, 
which contains an interior penalty term penalizing the jumps of the 
fluxes across the element edges/faces. Hence, it is a $C^0$ DG finite element method. 
In addition, we prove that the proposed  method is stable and converges with optimal 
order in a discrete $W^{2,p}$-type norm on quasi-uniform meshes provided that the 
polynomial degree of the finite element space is greater than or equal to two.

\begin{figure}
\tikzstyle{blockg} = [rectangle, draw, fill=white!20, 
text width=13em, text centered, rounded corners, minimum height=4em]
 \tikzstyle{blockg2} = [rectangle, draw, fill=white!20, 
 text width=17em, text centered, rounded corners, minimum height=4em]
\tikzstyle{line} = [draw, -latex']
\begin{center}
\begin{tikzpicture}[node distance = 7cm, auto]
 \node [blockg] (init) {{\small I.  Global stability estimate for PDEs  with constant coefficients\smallskip\\ 
    $\|w_h\|_{W^{2,p}_h(\Omega)}\les \|\Lcaloh w_h\|_{L^p_h(\Omega)}$}};
    \node[blockg,right of=init](localc){{\small II. Local stability estimate for PDEs with constant coefficients\smallskip\\ 
    $\|w_h\|_{W^{2,p}_h(B)}\les \|\Lcaloh w_h\|_{L^p_h(B^\prime)}$}};
    \node[blockg,below of=localc,node distance=3cm](localnd)
    {{\small III. Local stability estimate for PDEs in non-divergence form\smallskip\\
    $\|w_h\|_{W^{2,p}_h(B)}\les \|\Lcalh w_h\|_{L^p_h(B^\prime)}$}};
    \node[blockg2,left of=localnd](garding)
    {{\small IV. Global G\"arding-type inequality for PDEs in non-divergence form\smallskip\\
    $\|w_h\|_{W^{2,p}_h(\Omega)}\les \|\Lcalh w_h\|_{L^p_h(\Omega)}+\|w_h\|_{L^p(\Omega)}$}};
      \node[blockg,below of=garding,node distance=3cm](globalnd)
    {{\small V. Global stability estimate for PDEs in non-divergence form\smallskip\\
    $\|w_h\|_{W^{2,p}_h(\Omega)}\les \|\Lcalh w_h\|_{L^p_h(\Omega)}$}};
\path [line] (init) -- (localc);
\path [line] (localc) -- (localnd);
\path [line] (localnd) -- (garding);
\path [line] (garding) -- (globalnd);                
\end{tikzpicture}
\end{center}
\caption{\label{ProofFig}Outline of the convergence proof.}
\end{figure}
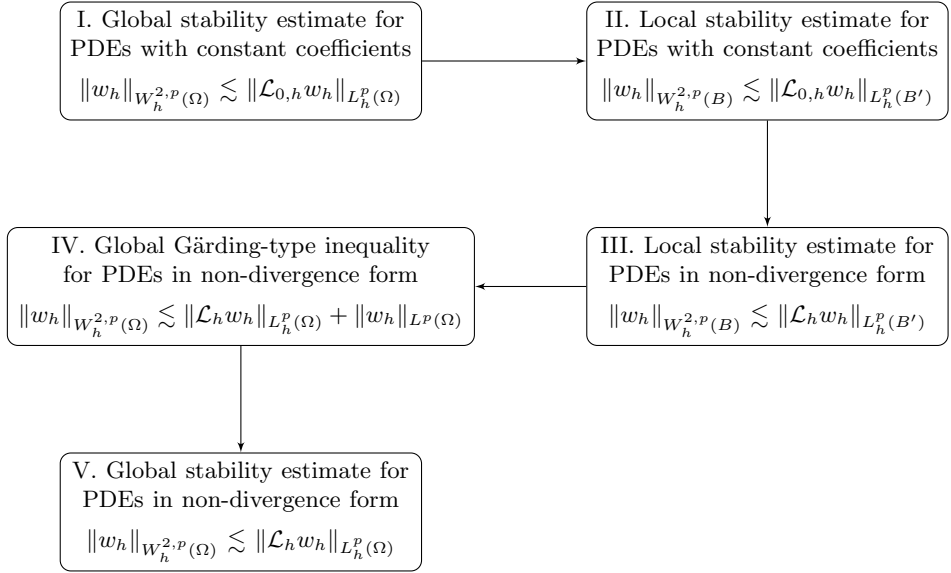

While the formulation and implementation of the finite element method is relatively
simple, the convergence analysis is quite involved, and it requires several nonstandard 
arguments and techniques. The overall strategy in the convergence analysis is
to mimic, at the discrete level, the stability analysis of strong solutions of PDEs
in non-divergence form (see \cite[Section 9.5]{Gilbarg_Trudinger01}).
Namely, we exploit the fact that locally, the finite element discretization
is a perturbation of a discrete elliptic operator in divergence form with
constant coefficients; see Lemma \ref{operatorDiffForm}.
The first step of the stability argument is to establish a discrete
Calderon-Zygmund-type estimate for the Lagrange finite element
discretization of the elliptic operator in \eqref{problem} with constant coefficients, 
which is equivalent to a global inf-sup condition for the discrete operator.
The second step is to prove a local version of the global estimate and inf-sup condition.
With these results in hand, local stability estimate for the proposed $C^0$ DG discretization
of \eqref{problem} can be easily obtained. We then glue these local stability estimates
to obtain a global G\"arding-type inequality.  Finally, to circumvent the lack of a (discrete)
maximum principle which is often used in the PDE analysis, we use a nonstandard duality
argument to obtain a global inf-sup condition for the proposed $C^0$ DG discretization for
problem \eqref{problem}.  See Figure \ref{ProofFig} for an outline
of the convergence proof.  Since the method is linear and consistent, the stability
estimate naturally leads to the well-posedness of the method and the energy norm error estimate.

The organization of the paper is as follows.  In Section \ref{sec-2}
the notation is set, and some preliminary results are given. Discrete $W^{2,p}$ 
stability properties, including a discrete Calderon-Zygmund-type estimate, of finite 
element discretizations of PDEs with constant coefficients are established. 
In Section \ref{sec-3}, we present the motivation and the formulation
of our $C^0$ discontinuous finite element method for problem \eqref{problem}.
Mimicking the PDE analysis from \cite{Gilbarg_Trudinger01} at the discrete level, 
we prove a discrete $W^{2,p}$ stability estimate for the discretization operator. 
In addition, we derive an optimal order error estimate in a discrete $W^{2,p}$-norm. 
Finally, in Section \ref{sec-4}, we give several numerical experiments 
which test the performance of the proposed $C^0$ DG finite element method and 
validate the convergence theory.

\section{Notation and preliminary results}\label{sec-2}

\subsection{Mesh and space notation}\label{sec-2.1}
Let $\Ome\subset \mathbb{R}^n$ be an bounded open domain. We shall use $D$ to denote a generic subdomain 
of $\Ome$ and $\p D$ denotes its boundary. $W^{s,p}(D)$ denotes the standard Sobolev spaces for $s\geq 0$ 
and $1\leq p\leq \infty$, $W^{0,p}(D)=L^p(D)$ and $W^{s,p}_0(\Ome)$ to denote the subspace 
of $W^{s,p}(\Ome)$ consisting functions whose traces vanish up to order $s-1$ on $\p\Ome$. 
$(\cdot, \cdot)_D$ denotes the standard inner product on $L^2(D)$ and $(\cdot, \cdot):=(\cdot,\cdot)_\Ome$.  
To avoid the proliferation of constants, we shall use the notation $a\les b$ to represent the 
relation $a\le Cb$ for some constant $C>0$ independent of mesh size $h$. 

Let $\mct:=\cTho$ be a quasi-uniform, simplical, and conforming triangulation 
of the domain $\Ome$.  Denote by $\mcei$ the set of interior edges in $\mct$, $\mceb$ the set 
of boundary edges in $\mct$, and $\mce = \mcei\cup \mceb$, the set of all edges in $\mct$.
We define the jump and average of a vector function $\bw$ 
on an interior edge $e=\p T^+\cap \p T^-$ as follows:
\begin{align*}
\jump{\bw}\big|_e&=\bw^{+}\cdot \bn_+ \big|_e+\bw^{-}\cdot \bn_-\big|_e, \\
\avg{\bw}\big|_e&=\frac12\Bigl(\bw^{+}\cdot \bn_+ \big|_e -\bw^{-}\cdot \bn_-\big|_e \Bigr), 
\end{align*}
where $\bw^{\pm}=\bw\big|_{T^{\pm}}$ and $\bn_\pm$ is the outward
unit normal of $T^\pm$.  

For a normed linear space $X$, we denote by $X^*$ its dual and $\bl\cdot,\cdot\br$ the pairing 
between $X^*$ and $X$. The Lagrange finite element space with respect to the triangulation is given by
\begin{align} \label{VhDef}
V_h :=\bigl\{v_h\in H^1_0(\Omega):\ v_h|_T\in \mathbb{P}_k(T)\ \forall T\in \mct \bigr\},
\end{align}
where $\mathbb{P}_k(T)$ denotes the set of polynomials with total degree not exceeding $k\  (\ge 1)$ on $T$. 
We also define the piecewise Sobolev space with respect to the mesh $\mct$
\begin{alignat*}{2}
&W^{s,p}(\mct) := \prod_{T\in \mct} W^{s,p}(T),  &&\qquad W^{(p)}_h:=W^{2,p}(\mct)\cap W_0^{1,p}(\Omega),  \\
&L^p_h(\mct):= \prod_{T\in \mct} L^p(T), &&\quad W^{s,p}_h(D) := W^{s,p}(\mct)\big|_D, 
\qquad L^p_h(D):=L^p(\mct)\big|_D.
\end{alignat*}
For a given subdomain $D\subseteq \Ome$, we also define $V_h(D)\subseteq V_h$ 
and $W_h^{(p)}(D)\subseteq W_h^{(p)}$ as the subspaces that vanish outside of $D$ by
\begin{align*}
V_h(D):&= \big\{v\in V_h;\, v|_{\Omega\backslash D}=0\big\},\qquad
W_h^{(p)}(D):= \big\{v\in W_h^{(p)};\ v|_{\Omega\backslash D}=0\big\}.
\end{align*}
We note that $V_h(D)$ is non-trivial for ${\rm diam}(D)>2h$.

Associated with $D\subseteq \Ome$, we define a semi-norm on $W^{2,p}_h(D)$ for $1< p<\infty$
\begin{align} \label{H2normB} 
\|v\|_{W^{2,p}_h(D)} &=\|D^2_h v\|_{L^p(D)}
+ \Big(\sum_{e\in \mcei}  h_e^{1-p}\bigl\|\jump{\nab v}\bigr\|_{L^p(e\cap \overline{D})}^p\Big)^{\frac{1}p}.
\end{align}
Here, $D^2_h v\in L^2(\Omega)$ denotes the piecewise Hessian matrix of $v$, i.e., 
$D^2_h v|_T = D^2 v|_T$ for all $T\in \mct$.  

Let $\mathcal{Q}_h: L^p(\Ome)\to V_h$ be the $L^2$ projection defined by
\begin{equation}\label{L2proj}
\bigl( \mathcal{Q}_h w, v_h \bigr) = \bigl(w, v_h \bigr)
\qquad \forall w\in L^2(\Ome),\, v_h \in V_h.
\end{equation}
It is well known that \cite{CT87} $\mathcal{Q}_h$ satisfies
for any $w\in W^{m,p}(\Omega)$ 
\begin{equation}\label{L2projection}
\| \mathcal{Q}_h w \|_{W^{m,p}(\Omega)} \les \|w\|_{W^{m,p}(\Omega)} \qquad m=0,1;\, 1<p<\infty.
\end{equation}

For any domain $D\subseteq\Ome$ and any $w\in L^p_h(D)$, we also introduce the following 
mesh-dependent semi-norm
\begin{align} \label{L2norm}
\|w\|_{L^p_h(D)}:= \sup_{0\neq v_h\in V_h(D)} \frac{\bigl(w, v_h \bigr)_D}{\|v_h\|_{L^{p^\prime}(D)}},
\qquad\mbox{where}\quad \frac{1}{p} +\frac{1}{p'}=1.
\end{align}
By \eqref{L2proj}, it is easy to see that $\|\cdot\|_{L^p_h(D)}$ is a norm on $V_h(D)$.
Moreover by \eqref{L2projection}
\begin{align}\label{norm_equiv}
\|w_h\|_{L^p(\Omega)} 
&= \sup_{v\in L^{p^\prime}(\Omega)} \frac{(w_h,v)}{\|v\|_{L^{p^\prime}(\Omega)}}
= \sup_{v\in L^{p^\prime}(\Omega)} \frac{(w_h,\cQ v)}{\|v\|_{L^{p^\prime}(\Omega)}}\\
&\nonum \les  \sup_{v\in L^{p^\prime}(\Omega)} \frac{(w_h,\cQ v)}{\|\cQ v\|_{L^{p^\prime}(\Omega)}}
 \le \|w_h\|_{L^p_h(\Omega)}\qquad \forall w_h\in V_h.
\end{align}

\subsection{Some basic properties of $W^{(p)}_h$ functions}\label{sec-2.2}

In this subsection we  cite or prove some basic properties of the broken Sobolev functions
in $W^{(p)}_h$, and in particular, for piecewise polynomial functions. These results, which have 
independent interest in themselves, will be used repeatedly in the later sections. We begin with 
citing a familiar trace inequality followed by proving an inverse inequality.

\begin{lem}[\cite{Brenner}]\label{TraceLemma}
For any $T\in \mct$, there holds
\begin{align}\label{TraceLine}
\|v\|_{L^p(\p T)}^p\les \big(h^{p-1}_T \|\nab v\|_{L^p(T)}^p 
+ h_T^{-1} \|v\|^p_{L^p(T)}\big)\qquad \forall v\in W^{1,p}(T)
\end{align}
for any $p\in (1,\infty)$.
Therefore by scaling, there holds
\begin{equation}\label{scalingTrace}
\sum_{e\in \mce^I} h_e \|v\|_{L^p(e\cap \bar{D})}^p \les 
\begin{cases}
\|v\|_{L^p({D})}^p  &\quad \forall v\in V_h(D),\\
\|v\|_{L^p({D})}^p + h^p \|\nabla v\|_{L^p({D})}^p &\quad\forall v\in W^{(p)}_h(D). 
\end{cases}
\end{equation}
\end{lem}

\begin{lem}\label{inverselem}
For any $v_h\in \vh,\ D\subseteq\Ome$, and $1< p<\infty$, there holds
\begin{align}\label{inverse}
\|v_h\|_{W^{2,p}_h(D)}\les h^{-1}\|v_h\|_{W^{1,p}({D}_h)},
\end{align}
where
\begin{align}\label{Dhdef}
{D}_h = \{x\in \Omega:\ {\rm dist}(x,D)\le h\}.
\end{align}
\end{lem}
\begin{proof}
By \eqref{H2normB}, \eqref{TraceLine} and inverse estimates \cite{Ciarlet78,Brenner}, we have
\begin{align*}
&\|v_h\|_{W^{2,p}_h(D)} 
= \|D^2_h v_h\|_{L^p(D)} + \Big(\sum_{e\in \mce^I} 
h_e^{1-p} \big\|\jump{\nab v_h}\big\|_{L^p(e\cap \bar{D})}^p\Big)^{\frac{1}{p}}\\
&\quad
\les  \|D^2_h v_h\|_{L^p(D)} + \mathop{\sum_{T\in \mct}}_{T\subset {D}_h} 
\Big(h_T^{1-p} \big(h_T^{p-1} \|D^2 v_h\|_{L^p(T)}^p 
+h_T^{-1} \|\nab v_h\|_{L^p(T)}^p\big)\Big)^{\frac{1}{p}}\\
&\quad
\les  h^{-1}\|v_h\|_{W^{1,p}({D}_h)}.
\end{align*}
\hfill
\end{proof}

The next lemma states a very simple fact about the 
discrete $W^{2,p}$ norm  on $W^{2,p}_h(\Ome)$.

\begin{lem}\label{discretenormestimates}
For any $1< p< \infty$, there holds 
\begin{align}\label{normestimate}
&\|\varphi\|_{{W}^{2,p}_h(\Ome)}\le  
\|\varphi\|_{W^{2,p}(\Ome)} \qquad \forall \varphi\in W^{2,p}(\Ome).
\end{align}
\end{lem}
%

Next, we state some super-approximation results of the nodal interpolant with respect
to the discrete $W^{2,p}$ semi-norm.  The derivation of the following results is 
standard \cite{Schatz98}, but for completeness we give the proof in Appendix \ref{AppendixA}
\begin{lem}\label{Superlem}
Denote by $I_h:C^0(\overline{\Omega})\to V_h$
 the nodal interpolant onto $V_h$.
Let $\eta \in C^\infty(\Omega)$ with $|\eta|_{W^{j,\infty}(\Omega)}\les d^{-j}$
for $0\le j\le k$.  Then for each $T\in \mct$ with $h\le d\le 1$, there holds
\begin{align}
\label{SuperLine2}
h^m \|\eta v_h-I_h (\eta v_h)\|_{W^{m,p}(D)}
&\les \frac{h}{d}\|v_h\|_{L^p({D}_h)}\qquad \mbox{for } m=0,1,\\
\label{SuperLine1}
\|\eta v_h-I_h (\eta v_h)\|_{W^{2,p}_h(D)}&\les \frac{1}{d^2} \|v_h\|_{W^{1,p}({D}_h)},
\end{align}
Moreover, if $k\ge 2$,  there holds
\begin{align}\label{SuperLine3}
\|\eta v_h-I_h (\eta v_h)\|_{W^{2,p}_h(D)}\les \frac{h}{d^3} \|v_h\|_{W^{2,p}(D_h)}.
\end{align}
Here,  $D\subset {D}_h\subset \Omega$ satisfy the conditions in Lemma {\rm \ref{inverselem}}.
\end{lem}

To conclude this subsection, we state and prove a discrete Sobolev interpolation estimate.

\begin{lem}\label{DiscreteInterp}
There holds for all $1< p<\infty$,
\begin{align*}
\|\nab w\|_{L^p(\Omega)}^2 \les \|w\|_{L^p(\Omega)} \|w\|_{W^{2,p}_h(\Omega)}
\qquad \forall w\in W_h^{(p)}.
\end{align*}
\end{lem}

\begin{proof}
Writing $\|\nab w\|_{L^p(\Omega)}^p = \int_\Ome |\nab w|^{p-2} \nab w\cdot \nab w\, dx$
and integrating by parts, we find
\begin{align}\nonum
&\|\nab w \|_{L^p(\Omega)}^p  
 = - \int_\Ome \Big(|\nab w|^{p-2} \Del w
+ (p-2)|\nab w|^{p-4} (D^2_h w \nab w)\cdot \nab w\Big) w\, dx\\
&\nonum\hskip 1in
+ \sum_{e\in \mce^I} \int_e \jump{|\nab w|^{p-2} \nab w} w\, ds\\
&\label{DinterpLine0}
\quad
\les \sum_{T\in \mct} \int_T |\nab w|^{p-2} |D^2_h w| |w|\, dx
+ \sum_{e\in \mce^I} \int_e \jump{|\nab w|^{p-2} \nab w} w\, ds.
\end{align}
To bound the first term in \eqref{DinterpLine0} we apply H\"older's inequality to obtain
\begin{align}\label{DinterpLine1}
\int_\Omega |\nab w|^{p-2} |D^2_h w| |w|\, dx
&\le \big\||\nab w|^{p-2}\big\|_{L^{\frac{p}{p-2}(\Ome)}}\|D^2 w\|_{L^p(\Ome)}\|w\|_{L^p(\Ome)}\\
&\nonum= \|\nab w\|_{L^p(\Ome)}^{p-2} \|D^2_h w\|_{L^p(\Ome)} \|w\|_{L^p(\Ome)}.
\end{align}
Likewise, by Lemma \ref{TraceLemma} we have
\begin{align}\label{DinterpLine2}
&\sum_{e\in \mce^I} \int_e \jump{|\nab w|^{p-2} \nab w} w\, ds\\
&\ \nonum
\le \sum_{e\in \mce^I} \Big(h_e^{\frac{1}{p}}\big\|\jump{\nab w}\big\|_{L^p(e)}\Big)^{p-2} 
\Big(h_e^{\frac{1-p}{p}}\|\jump{\nab w}\|_{L^p(e)}\Big) \Big(h_e^{\frac{1}{p}}\|w\|_{L^p(e)}\Big)\\
&\ \nonum
\les  \|\nab w\|^{p-2}_{L^p(\Ome)} \|w\|_{L^p(\Ome)} 
\Big(\sum_{e\in \mce^I} h_e^{1-p} \big\|\jump{\nab w}\big\|_{L^p(e)}^p\Big)^{\frac{1}{p}}\\
&\ \nonumber
\les  \|\nab w\|^{p-2}_{L^p(\Omega)} \|w\|_{L^p(\Omega)}\|w\|_{W^{2,p}_h(\Omega)}.
\end{align}
Combining \eqref{DinterpLine0}--\eqref{DinterpLine2} 
we obtain the desired result. The proof is complete.
\end{proof}

\subsection{Stability estimates for auxiliary PDEs with constant coefficients}\label{sec-2.3}

In this subsection, we consider a special case of \eqref{problem1} when the coefficient
matrix is a constant matrix, $A(x)\equiv A_0\in \mathbb{R}^{n\times n}$. We introduce
the finite element approximation (or projection) $\Lcaloh$ of $\Lcalo$
on $V_h$ and extend $\Lcaloh$ to the broken Sobolev space $W^{(p)}_h$.  
We then establish some stability results for the operator $\Lcaloh$. 
These 
stability results will play an important role in our convergence analysis of the proposed 
$C^0$ DG finite element method in Section \ref{sec-3}. 

Let $A_0\in \mathbb{R}^{n\times n}$ be a 
positive definite matrix and set
\begin{alignat}{2} \label{mathcalLdef}
\Lcalo w &:=-A_0:D^2 w = -\nab \cdot \(\An\nab w\).
\end{alignat}
The operator $\mathcal{L}_0$ induces the following bilinear form:
\begin{equation}\label{bilinear_form_0}
a_0(w,v):= \bl  \Lcalo w, v \br = \int_\Omega \An \nab w\cdot\nab  v\, dx 
\qquad \forall w,v\in H^1_0(\Ome),
\end{equation} 
and the Lax-Milgram Theorem (cf. \cite{EvansBook}) implies that $\Lcalo^{-1} :H^{-1}(\Omega)\to H^1_0(\Omega)$
exists and is bounded. Moreover, if $\p\Ome\in C^{1,1}$, the Calderon-Zygmund theory
(cf. \cite[Chapter 9]{Gilbarg_Trudinger01}) infers that $\Lcalo^{-1}:L^{p}(\Omega)
\to W^{2,p}(\Omega)\cap W^{1,p}_0(\Omega)$ exists and there holds 
\begin{align}\label{Strongstability_new}
\|\Lcalo^{-1} \varphi\|_{W^{2,p}(\Ome)}\les \|\varphi \|_{L^{p}(\Ome)}
\qquad \forall \varphi\in L^p(\Ome).
\end{align}
Equivalently,
\begin{align}\label{Strongstability}
\|w\|_{W^{2,p}(\Ome)}\les \big\|\Lcalo w \big\|_{L^{p}(\Ome)}
\qquad \forall w\in W^{2,p}\cap W^{1,p}_0(\Ome).
\end{align}

The bilinear form naturally leads to a finite element approximation (or projection) 
of $\mathcal{L}_0$ on $V_h$, that is, we define the operator 
$\Lcaloh: V_h\to V_h$ by
\begin{align} \label{add9c}
\bigl(\Lcaloh w_h,v_h \bigr) :=a_0(w_h,v_h) \qquad\forall v_h, w_h\in V_h.
\end{align}

\begin{remark}
When $A=I$, the identity matrix, $\Lcaloh$ is exactly the finite element 
the discrete Laplacian that is,  $\Lcaloh=-\Delta_h$. By finite element theory {\rm \cite{Brenner}},
we know that $\Lcaloh:V_h\to V_h$ is one-to-one and onto, and therefore $\Lcaloh^{-1}:V_h\to V_h$ exists.
\end{remark}

Recall the following DG integration by parts formula:
\begin{align} \label{add9}
\int_\Ome \tau \cdot \nabla_h v \, dx &=-\int_\Ome (\nabla_h\cdot\tau) v\, dx 
+\sum_{e\in \mathcal{E}_h^I} \Bigl( \int_e \jump{\tau} \avg{v} \, ds \\
&\hskip 0.8in
+\int_e  \avg{\tau}\cdot \jump{v} \, ds \Bigr)
+ \sum_{e\in \mathcal{E}_h^B} \int_e ( \tau \cdot n_e)  v \, ds, \nonumber
\end{align}
which holds for any piecewise scalar--valued function $v$ and vector--valued function
$\tau$.  Here, $\nabla_h$ is defined piecewise, i.e., $\nabla_h|_T= \nabla|_T$
for all $T\in \mct$. For any $w_h, v_h\in V_h$, using \eqref{add9} with
$\tau=A_0\nabla w_h$, we obtain
\begin{align} \label{add9b}
a_0(w_h,v_h) 
=-\int_\Omega \big(A_0:D^2_h w_h\big)v_h\, dx 
+ \sum_{e\in \mce^I} \int_e \jump{A_0\nab w_h}v_h\, ds.
\end{align}
We note that the above new form  of $a_0(\cdot,\cdot)$ is 
not well defined on $H^1_0(\Omega)\times H^1_0(\Omega)$.
However, it is well defined on 
$W^{(p)}_h\times W^{(p^\prime)}_h$ with $\frac{1}{p}+\frac{1}{p^\prime}=1$.
Hence, we can easily extend the domain of the 
operator $\Lcaloh$ to broken Sobolev space $W^{(p)}_h$. 
Precisely, (abusing the notation)
we define $\Lcaloh:W^{(p)}_h\to (W^{(p^\prime)}_h)^*$ to be the operator induced by 
the bilinear form $a_0(\cdot. \cdot)$ on $W^{(p)}_h\times W^{(p^\prime)}_h$, namely,
\begin{alignat}{2} \label{Ldef}
\bl \Lcaloh w,v\br := a_0(w, v)\qquad \forall w\in W^{(p)}_h,\ v\in W^{(p^\prime)}_h.
\end{alignat}

A key ingredient in the convergence analysis of our finite element methods
for PDEs in non--divergence form is to establish  global and local discrete 
Calderon--Zygmund-type estimates similar to \eqref{Strongstability} for $\Lcaloh$.  
These results are presented in the following two lemmas.

\begin{lem}\label{StabilityLem}
There exists $h_0>0$ such that for all $h\in (0,h_0)$ there holds
\begin{align}\label{stability}
\|w_h \|_{W_h^{2,p}(\Ome)} \les  \|\Lcaloh w_h\|_{L^p(\Ome)}\qquad \forall w_h\in V_h.
\end{align}
\end{lem}

\begin{proof}
First note that \eqref{stability} is equivalent to
\begin{align}\label{stability_new}
\big\|\Lcaloh^{-1} \varphi_h \big\|_{W_h^{2,p}(\Ome)} 
\les  \| \varphi_h\|_{L^p(\Ome)}\qquad \forall \varphi_h\in V_h.
\end{align}

For any fixed $\varphi_h\in V_h$, let $w:=
\Lcalo^{-1} \varphi_h\in W^{2,p}(\Ome)\cap W^{1,p}_0(\Ome)$ 
and $w_h:=\Lcaloh^{-1} \varphi_h\in V_h$. Therefore, $w$ and $w_h$, 
respectively, are the solutions of the following two problems:
\begin{align}\label{add4}
a_0(w, v)=(\varphi_h, v) \quad \forall v\in H^1_0(\Omega),
\qquad a_0(w_h, v_h)=(\varphi_h, v_h) \quad \forall v_h\in V_h,
\end{align}
and thus, $w_h$ is the elliptic projection of $w$.

By \eqref{Strongstability} we have 
\begin{equation}\label{add3}
\|w\|_{W^{2,p}(\Ome)} \les \|\varphi_h\|_{L^p(\Ome)}.
\end{equation}
Using well--known $L^p$ finite element estimate results \cite[Theorem 8.5.3]{Brenner},
finite element interpolation theory, and \eqref{add3} we obtain that there exists $h_0>0$ such 
that for all $h\in (0, h_0)$
\begin{align}\label{add5}
\|w-w_h\|_{W^{1,p}(\Ome)} \les \|w-I_h w\|_{W^{1,p}(\Ome)}
\les h \|w\|_{W^{2,p}(\Ome)} \les h \|\varphi_h\|_{L^p(\Ome)}.
\end{align}

It follows from the triangle inequality, an inverse inequality
(see Lemma \ref{inverselem}), the stability of $I_h$, \eqref{add3} and \eqref{add5} that 
\begin{align*}
\|w-w_h\|_{W^{2,p}_h(\Ome)} &\les  \|w- I_hw\|_{W^{2,p}_h(\Ome)} 
+ \|I_hw- w_h\|_{W^{2,p}_h(\Ome)} \\
&\les \|w\|_{W^{2,p}(\Ome)} + h^{-1} \|I_h w-w_h\|_{W^{1,p}(\Ome)} \\
&\les \|\varphi_h\|_{L^p(\Ome)} + h^{-1} \|I_h w-w\|_{W^{1,p}(\Ome)}
+ h^{-1} \|w-w_h\|_{W^{1,p}(\Ome)} \\
&\les \|\varphi_h\|_{L^p(\Ome)}.
\end{align*}
Thus,
\begin{equation*}
\|w_h\|_{W^{2,p}_h(\Ome)}
\les \|w-w_h\|_{W^{2,p}_h(\Ome)} +\|w\|_{W^{2,p}(\Ome)} 
\les  \|\varphi_h\|_{L^p(\Ome)},
\end{equation*}
which yields \eqref{stability_new}, and hence, \eqref{stability}.
\end{proof}

\begin{lem}\label{localstabilitylem}
For $\xn\in \Ome$ and  $R>0$, define
\begin{align} \label{BallRDef}
B_{R}(\xn) := \{x\in \Ome:\ |x-\xn|< R\} \subset \Ome.
\end{align}
Let $R^\prime = R+d$ with $d\ge 2h$.
Then there holds
\begin{align}\label{localstability}
\|w_h\|_{W^{2,p}_h(B_R(\xn))} \les  \big\|\Lcaloh w_h\|_{L^p_h(B_{R^\prime}(\xn))}\qquad \forall w_h\in V_h(B_R(\xn)),
\end{align}

\end{lem}
\begin{proof}
To ease notation, set 
$B_R:=B_R(\xn)$ and $B_{R^\prime} :=B_{R^\prime}(\xn)$.
Recalling \eqref{norm_equiv}, we have 
by Lemma \ref{StabilityLem},
\begin{align*}
\|w_h\|_{W^{2,p}_h(B_{R})} 
&= \|w_h\|_{W^{2,p}_h(\Omega)}
\les \|\Lcaloh w_h\|_{L^p(\Omega)} 
\les \|\Lcaloh w_h\|_{L^p_h(\Omega)}
= \sup_{v_h\in V_h} \frac{a_0(w_h,v_h)}{\|v_h\|_{L^{p^\prime}(\Omega)}}.
\end{align*}
Set $R^{\prime\prime} = (R+R^\prime)/2$, so that $R< R''< R'$. Denote by $\ind$ the indicator function 
of $B_{R^{\prime\prime}}:=B_{R^{\prime\prime}}(\xn)$.  Since $w_h = 0$ on $\Omega\backslash B_R$, we have
\begin{align*}
a_0(w_h,v_h) = a_0(w_h,\ind v_h) = a_0(w_h,I_h(\ind v_h))\qquad \forall v_h\in V_h.
\end{align*}
Moreover, we have $I_h (\ind v_h)\in V_h(B_{R^\prime})$ and
\begin{align*}
\|I_h (\ind v_h)\|_{L^{p^\prime}(B_{R^\prime})} = \|I_h (\ind v_h)\|_{L^{p^\prime}(\Omega)}\les \|\ind v_h\|_{L^{p^\prime}(\Omega)}\les \|v_h\|_{L^{p^\prime}(\Omega)}.
\end{align*}
Consequently, 
\begin{align*}
\|w_h\|_{W^{2,p}_h(B_R)} 
&\les \sup_{v_h\in V_h} \frac{a_0(w_h,I_h (\ind v_h))}{\|I_h (\ind v_h)\|_{L^{p^\prime}(B_{R^\prime})}}
\le \sup_{v_h\in V_h(B_{R^\prime})} \frac{a_0(w_h,v_h)}{\|v_h\|_{L^{p^\prime}(B_{R^\prime})}}\\
& = \sup_{v_h\in V_h(B_{R^\prime})} \frac{(\Lcaloh w_h,v_h)}{\|v_h\|_{L^{p^\prime}(B_{R^\prime})}} = \|\Lcaloh w_h\|_{L^p_h(B_{R^\prime})}.
\end{align*}
\hfill
\end{proof}

\section{$C^0$ DG finite element methods and convergence analysis}\label{sec-3}

\subsection{The PDE problem} \label{sec-3.1}
To make the presentation clear, we state the precise assumptions on the non-divergence form PDE 
problem \eqref{problem}.  Let $A\in [C^0(\overline{\Ome})]^{n\times n}$ be a 
positive definite matrix-valued function with
\begin{equation}\label{ellipticity}
\lambda |\xi|^2 A(x) \xi\cdot \xi \leq \Lambda |\xi|^2 \qquad 
\forall \xi\in \mathbb{R}^n,\, x\in \overline{\Ome}
\end{equation}
and constants $0<\lambda\le \Lambda<\infty.$
Under the above assumption, $\mathcal{L}$ is known to be uniformly elliptic,
hence, strong solutions (i.e., $W^{2,p}$ solutions) of problem \eqref{problem}
must satisfy the {\em Aleksandrov maximum principle} 
\cite{Gilbarg_Trudinger01,EvansBook,HanLinBook}. 

By the $W^{2,p}$ theory for the second order non-divergence form uniformly 
elliptic PDEs \cite[Chapter 9]{Gilbarg_Trudinger01}, we know that if 
$\p\Ome\in C^{1,1}$, for any $f\in L^p(\Ome)$ with $1<p<\infty$, there exists 
a unique strong solution $u\in W^{2,p}(\Omega)\cap W^{1,p}_0(\Omega)$ to 
\eqref{problem} satisfying
\begin{equation}\label{CZ_estimate}
\|u\|_{W^{2,p}(\Omega)}\les \|f\|_{L^{p}(\Omega)}.
\end{equation}
Moreover, when $n=2$ and $p=2$, it is also known that 
\cite{Grisvard85,Fromm,Bernstein1910,LadyBook,BabuskaOsborn94}
the above conclusion holds if $\Ome$ is a convex domain.

For the remainder of the paper, we shall always 
assume that  $A\in [C^0(\overline{\Ome})]^{n\times n}$ is 
positive definite satisfying \eqref{ellipticity}, and problem \eqref{problem} has a unique 
strong solution $u$ which satisfies the Calderon-Zygmund estimate \eqref{CZ_estimate}.  

\subsection{Formulation of $C^0$ DG finite element methods}
The formulation of our $C^0$ DG finite element method for non-divergence
form PDEs is relatively simple, which is inspired by the finite element method
for divergence form PDEs and relied only on an unorthodox integration by parts. 

To motivate its derivation, we first look at how one would construct standard finite element 
methods for problem \eqref{problem} when the coefficient matrix 
$A$ belongs to $[C^1(\overline{\Ome})]^{n\times n}$.  In this case, 
since the divergence of $A$  (taken row-wise) is well defined,
we can rewrite the PDE \eqref{problem1} in divergence form as follows:
\begin{align}\label{DivergenceForm}
-\nab \cdot (A\nab u)+(\nab \cdot A)\cdot \nab u =f.
\end{align}
Hence, the original non-divergence form PDE is converted into a 
``diffusion-convection equation" with the ``diffusion coefficient" 
$A$ and the ``convection coefficient" $\nabla \cdot A$.

A standard finite element method for problem \eqref{DivergenceForm} is readily 
defined as seeking $u_h\in V_h$ such that 
\begin{align}\label{SFEM}
\int_\Omega (A\nab u_h)\cdot \nab v_h\, dx 
+\int_\Omega (\nab \cdot A)\cdot \nab u_h v_h\, dx 
= \int_\Omega f v_h\, dx\quad \forall v_h\in V_h.
\end{align}

Now come back to the case where $A$ only belongs to
$[C^0(\overline{\Omega})]^{n\times n}$. In our setting,
the formulation \eqref{SFEM} is not viable any more because $\nabla\cdot A$ 
does not exist as a function. 
To circumvent this issue, we apply the DG integration by parts formula \eqref{add9}
to the first term on the left-hand side of \eqref{SFEM} with $\tau=A\nabla u_h$
and $\nabla$ in \eqref{SFEM} is understood piecewise, we get
\begin{align} \label{IntroMethod}
-\int_\Omega \big(A:D^2_h u_h\big)v_h\, dx 
+ \sum_{e\in \mce^I} \int_e \jump{A\nab u_h}v_h\, ds 
= \int_\Omega fv_h\, dx\qquad \forall v_h\in V_h.
\end{align}
Here we have used the fact that $\jump{v_h}=0$ and $v_h|_{\p \Omega}=0$.

No derivative is taken on $A$ in \eqref{IntroMethod}, so each of the terms 
is well defined on $V_h$. This indeed yields the $C^0$ DG formulation of this paper.

\begin{definition}\label{def1}
The $C^0$ discontinuous Galerkin (DG) finite element method is defined by seeking 
$u_h\in V_h$ such that
\begin{align} \label{MethodShort}
a_h(u_h,v_h) = (f, v_h) \qquad \forall v_h\in V_h,
\end{align}
where
\begin{align}\label{BLF_1}
a_h(w_h,v_h) &:= -\int_\Omega \big(A:D^2_h w_h\big)v_h\, dx 
+ \sum_{e\in \mce^I} \int_e \jump{A\nab w_h}v_h\, ds, \\
(f, v_h) &:= \int_\Omega f v_h\, dx\qquad \forall v_h\in V_h. \label{BLF_2} 
\end{align} 

\end{definition}

A few remarks are given below about the proposed $C^0$ DG finite element method.

\begin{remark}
(a) The above method is also defined for $A\in [L^\infty(\Ome)]^{n\times n}$ and
no {\em a priori} knowledge of the location of the singularities of $A$ are required
in the meshing procedure.

(b) The $C^0$ DG finite element method \eqref{MethodShort} is a primal method 
with the single unknown $u_h$.  It can be implemented on current finite element 
software supporting element boundary integration.

(c) From its derivation we see that  \eqref{MethodShort} is equivalent to the standard 
finite element method \eqref{SFEM} provided $A$ is smooth.  In addition, if $A$ is 
constant then \eqref{MethodShort} reduces to 
\begin{align*}
a_0(u_h,v_h) = (f, v_h) \qquad\forall v_h\in V_h.
\end{align*}
This feature will be crucially used in the convergence analysis later.

(d) In the one-dimensional and piecewise linear case (i.e., $n=1$ and $k=1$), the 
method \eqref{MethodShort} on a uniform mesh $\{x_i\}_{i=1}^N$ is equivalent to
\begin{align*}
A(x_i)\big(-c_{i-1}+2c_i-c_{i+1}\big) = h^2 f(x_i),
\end{align*}
where $u_h = \sum_{i=1}^N c_i \varphi^{(i)}_h$, and $\{\varphi_h^{(i)}\}_{i=1}^N$ 
represents the nodal basis for $V_h$.
\end{remark}

\subsection{Stability analysis and well-posedness theorem} \label{sec-3.3}
As in Section \ref{sec-2.3}, using the bilinear form $a_h(\cdot, \cdot)$ we 
can define the finite element approximation (or projection) $\Lcalh$ of $\Lcal$ on $V_h$,
that is, we define  
$\Lcalh: V_h\to V_h$ by
\begin{align} \label{FEM_operator}
\bigl( \Lcalh w_h,v_h \bigr):=a_h(w_h,v_h) \qquad \forall v_h,w_h \in V_h.
\end{align}
Trivially, \eqref{MethodShort} can be rewritten as: Find $u_h\in V_h$ such that
\[
\bigl( \Lcalh u_h,v_h \bigr)= (f,v_h) \qquad \forall v_h\in V_h.
\]
Similar to the argument for $\Lcaloh$, 
we can  extend the domain of $\Lcalh$ to the broken Sobolev space $W^{(p)}_h$, 
that is, (abusing the notation) we define $\Lcalh: W^{(p)}_h \to (W^{(p^\prime)}_h)^*$ by 
\begin{align} \label{extended_FEM_operator}
\bl \Lcalh w,v\br:=a_h(w,v) \qquad \forall w\in W_h^{(p)},\ v\in W_h^{(p^\prime)}.
\end{align}

The main objective of this subsection is to establish a $W^{2,p}_h$ stability 
estimate for the operator $\Lcalh$ on the finite element space $V_h$.
From this result, the existence, uniqueness and error estimate 
for \eqref{MethodShort} will naturally follow. The stability proof relies 
on several technical estimates which we derive below.
Essentially, the underlying strategy, known as a perturbation argument
in the PDE literature, is to treat the operator $\Lcalh$ locally as a perturbation of
a stable operator with constant coefficients.  The following lemma quantifies this statement.

\begin{lem}\label{operatorDiffForm}
For any $\delta>0$, there exists $R_\delta>0$ and $h_\delta>0$ such that for any 
$x_0\in \Ome$ with $A_0=A(x_0)$ 
\begin{align} \label{continuity}
\|(\Lcalh-\Lcaloh)w\|_{L^p_h(B_{R_\delta}(x_0))}
\les \delta \|w\|_{W^{2,p}_h(B_{R_\delta}(x_0))}\quad \forall w\in W^{(p)}_h,\, 
\forall h\leq h_\delta.
\end{align}
\end{lem}
\begin{proof}
Since $A$ is continuous on $\overline{\Omega}$, it is uniformly continuous.
Therefore for every $\delta>0$ there exists $R_\delta>0$ such that if $x,y\in \Omega$
satisfy $|x-y|<R_\delta$, there holds $|A(x)-A(y)|<\delta$. Consequently
for any $x_0\in \Ome$
\begin{align}\label{AC101}
\|A-\An\|_{L^\infty(B_{R_\delta})}\le \delta
\end{align}
with $B_{R_\delta}:=B_{R_\delta}(x_0)$.

Set $h_\delta=\min\{h_0, \frac{R_\delta}4\}$ and consider $h\leq h_\delta$,
$w\in W^{(p)}_h$ and $v_h\in V_h(B_{R_\delta})$.
Since $(\Lcaloh-\Lcalh) w\in W^{(p)}_h$, it follows from \eqref{norm_equiv}, \eqref{add9b}, 
\eqref{BLF_1}, \eqref{AC101}, and \eqref{L2projection} that
\begin{align*}
&\bigl( (\Lcaloh-\Lcalh) w, v_h \bigr)\\
&\qquad
=-\int_{B_{R_\delta}} \((\An-A):D^2_h w\)  v_h\, dx 
+\sum_{e\in \mcei} \int_{e\cap \bar{B}_{R_\delta}} \jump{(\An-A) \nab w}{v_h}\, ds\\
&\qquad
\le \|A-\An\|_{L^\infty(B_{R_\delta})}\Bigl( \|D^2_h w\|_{L^p(B_{R_\delta})} \|v_h\|_{L^{p^\prime}(B_{R_\delta})}\\
&\qquad\qquad
+\Bigl({\sum_{e\in \mcei}} h_e^{1-2p} \big\|\jump{\nab w}\big\|_{L^p(e\cap \bar{B}_{R_\delta})}^p\Bigr)^{\frac{1}{p}}
\Bigl({\sum_{e\in \mcei}} h_e\|v_h\|_{L^{p^\prime}(e\cap \bar{B}_{R_\delta})}^{p^\prime}\Bigr)^{\frac{1}{p^\prime}}\Bigr)\\
&\qquad
\les \|A-\An\|_{L^\infty(B_{R_\delta})} \|w\|_{W^{2,p}_h(B_{R_\delta})} \|v_h\|_{L^{p^\prime}(B_{R_\delta})} 
%
\les \delta  \|w\|_{W^{2,p}_h(B_{R_\delta})} \|v_h\|_{L^{p^\prime}(B_{R_\delta})}.
\end{align*}
The desired inequality now follows from the definition of $\|\cdot\|_{L^p_h(B_{R_\delta})}$.
\end{proof}

\begin{lem}\label{LocalNonDivLemma}
There exists $R_1>0$ and $h_1>0$ such that for any $\xn\in \Omega$ 
\begin{align} \label{local_stability}
\|w_h\|_{W^{2,p}_h(B_{R_1}(x_0))}
\les  \|\Lcalh w_h\|_{L^{p}_h(B_{R_2}(x_0))} \quad \forall w_h\in V_h(B_{R_1}(x_0)),\,
\forall h\leq h_1,
\end{align}
with $R_2= 2R_1$.
\end{lem}

\begin{proof}
For $\delta_0>0$ to be determined below, let
$R_1=\frac12 R_{{\delta_0}}$ as in Lemma \ref{operatorDiffForm}.
Let $h_1 = \frac{R_1}{2}$
and set $B_i=B_{R_i}(x_0)$.
Then by
Lemmas \ref{localstabilitylem} and \ref{operatorDiffForm} 
with $d=R_1$ and $A_0=A(x_0)$, we have for any $w_h\in V_h(B_1)$
\begin{align*}
\|w_h\|_{W^{2,p}_h(B_1)} &\les  \|\Lcaloh w_h\|_{L^p_h(B_2)}
\le \|(\Lcaloh-\Lcalh) w_h\|_{L^p_h(B_2)}+\|\Lcalh w_h\|_{L^p_h(B_2)} \\
&\les \delta_0 \|w_h\|_{W^{2,p}_h(B_2)} + \|\Lcalh w_h\|_{L^p_h(B_2)}
 = \delta_0 \|w_h\|_{W^{2,p}_h(B_1)} + \|\Lcalh w_h\|_{L^p_h(B_2)}.
\end{align*}
For $\delta_0$ sufficiently small (depending only on $A$), we can kick back 
the first term on the right-hand side.  This completes the proof.
\end{proof}

\begin{lem}\label{LContLemma}
Let $R_1$ and $h_1$ be as in Lemma {\rm\ref{LocalNonDivLemma}}. For any $x_0\in \Ome$, there holds 
\begin{align} \label{reverse_stability}
\|\Lcalh w\|_{L^p_h(B_{R_1}(x_0))}\les  \|w\|_{W^{2,p}_h(B_{R_1}(x_0))} 
\qquad \forall w\in W^{(p)}_h,\, \forall h\leq h_1.
\end{align}
\end{lem}
\begin{proof}
Set $B_1=B_{R_1}(x_0)$. By the definition of $\Lcalh$, \eqref{norm_equiv},
\eqref{scalingTrace} and \eqref{L2projection}, we have for any 
$v_h\in V_h(B_1)$ 
\begin{align*}
(\Lcalh w,v) &=-\int_{B_1} (A:D^2_h w) v_h\, dx 
+ \sum_{e\in \mce^I} \int_{e\cap \bar{B}_1} \jump{A\nab w} v_h\, ds\\
&\les \|D^2 _h w\|_{L^p(B_1)} \|v_h\|_{L^{p^\prime}(B_1)} \\
&\qquad
+ \Bigl(\sum_{e\in \mce^I} h_e^{1-p} \big\|\jump{\nab w}\big\|_{L^p(e\cap \bar{B}_1)}^p\Bigr)^{\frac{1}{p}}
\Bigl(\sum_{e\in \mce^I} h_e \|v_h\|_{L^{p^\prime}(e\cap \bar{B}_1)}^{p^\prime}\Bigr)^{\frac{1}{p^\prime}}\\
&\les \Bigl( \|D^2 _h w\|_{L^p(B_1)} 
+ \Big(\sum_{e\in \mce^I} h_e^{1-p} 
\big\|\jump{\nab w}\big\|_{L^p(e)}^p\Big)^{\frac{1}{p}}\Bigr) \|v_h\|_{L^{p^\prime}(B_1)}\\
&\les   \|w\|_{W^{2,p}_h(B_1)} \|v_h\|_{L^{p^\prime}(B_1)}.
\end{align*}
The desired inequality now follows from the definition of $\|\cdot\|_{L^p_h(B_1)}$.\hfill
\end{proof}

\begin{lem}\label{localStabilityLemma}
Let $h_1$ be as in Lemma \ref{LocalNonDivLemma}.  Then there holds for $h\leq h_1$
\begin{align}\label{interiorstability}
\|w_h\|_{W^{2,p}_h(\Omega)}\les \|\Lcalh w_h\|_{L^p(\Omega)}
 + \|w_h\|_{L^p(\Omega)} \qquad \forall w_h\in V_h.
\end{align}
\end{lem}

\begin{proof}
We divide the proof into two steps.

{\em Step 1}: 
For any $x_0\in \Ome$, let $R_1$ and $h_1$ be as in Lemma \ref{LocalNonDivLemma}, 
let $R_2 = 2R_1$, $R_3 = 3R_1$, and set $B_i=B_{R_i}(x_0)$ for $i=0,1,2$.  
Let $\eta\in C^3(\Omega)$ be a cut-off function satisfying 
\begin{align}\label{etaprop}
0\le \eta\le 1,\quad \eta\big|_{B_1}=1,\quad\eta\big|_{\Ome\backslash B_{2}} 
= 0,\quad \|\eta\|_{W^{m,\infty}(\Omega)}=O(d^{-m})\quad m=0,1,2.
\end{align}

We first note that $\eta w_h\in W^{(p)}_h(B_2)$ and $I_h (\eta w_h)\in V_h(B_3)$ for any $w_h\in V_h$.
Therefore, by Lemmas \ref{Superlem} (with $d=R_1$) and \ref{LocalNonDivLemma}, we have
\begin{align*}
\|w_h\|_{W^{2,p}_h(B_1)}  
&= \|\eta w_h\|_{W^{2,p}_h(B_1)}
\le \|\eta w_h-I_h (\eta w_h)\|_{W^{2,p}_h(B_1)}+\|I_h (\eta w_h)\|_{W^{2,p}_h(B_1)}\\
 &\nonum\les \frac{1}{R_1^{2}} \|w_h\|_{W^{1,p}(B_2)}+ \|I_h (\eta w_h)\|_{W^{2,p}_h(B_1)}\\
 &\nonum\les  \frac{1}{R_1^{2}} \|w_h\|_{W^{1,p}(B_2)} + \|\Lcalh (I_h(\eta w_h))\|_{L^p_h(B_2)}\\
 &\nonum\les \frac{1}{R_1^2} \|w_h\|_{W^{1,p}(B_2)} + \|\Lcalh (\eta w_h)\|_{L^p_h(B_2)}
 +\|\Lcalh (\eta w_h-I_h(\eta w_h))\|_{L^p_h(B_2)}.
 \end{align*}
Applying Lemmas \ref{LContLemma} and \ref{Superlem}, we obtain
\begin{align} \label{LongLine1}
\|w_h\|_{W^{2,p}_h(B_1)}  
&\les \frac{1}{R_1^{2}} \|w_h\|_{W^{1,p}(B_2)} + \|\Lcalh (\eta w_h)\|_{L^p_h(B_2)} \\
&\hskip 1.2in \nonum 
+\|\eta w_h-I_h(\eta w_h)\|_{W^{2,p}_h(B_2)} \\
&\les \frac{1}{R_1^{2}} \|w_h\|_{W^{1,p}(B_3)} +  \|\Lcalh (\eta w_h)\|_{L^p_h(B_3)}.  \nonumber
\end{align}

To derive an upper bound of the last term in \eqref{LongLine1}, we write for $v_h\in V_h(B_3)$,
\begin{align*}
&\bigl(\Lcalh (\eta w_h), v_h\bigr) 
=-\int_{B_3} A:D^2_h (\eta w_h) v_h\, dx 
+ \sum_{e\in \mce^I} \int_{e\cap \bar{B}_3} \jump{A\nab(\eta w_h)} v_h\, ds\\
&=-\int_{B_3} \bigl( \eta A:D_h^2 w_h 
+ 2A\nabla\eta\cdot \nabla w_h + w_h A:D_h^2 \eta \bigr) v_h \, dx 
+\sum_{e\in \mce^I}  \int_{e\cap \bar{B}_3}  \jump{A\nab w_h} \eta  v_h\, ds \nonumber \\
&=\bigl( \Lcalh w_h,I_h (\eta  v_h)\bigr) - \int_{B_3} \bigl( 2A\nabla\eta\cdot \nabla w_h 
+ w_h A:D_h^2 \eta \bigr) v_h \, dx \nonumber \\
&\qquad\nonumber  -\int_{B_3} (A:D^2_h w_h)(\eta v_h- I_h (\eta v_h))\, dx
+\sum_{e\in \mce^I} \int_{e\cap \bar{B}_3} \jump{A\nab w_h} (\eta v_h-I_h (\eta v_h))\, ds.
\end{align*}
By H\"older's inequality, Lemmas \ref{TraceLemma}--\ref{inverselem},  \ref{Superlem},  
and \eqref{etaprop} we obtain
\begin{align*}
&\bigl( \Lcalh (\eta w_h),v_h \bigr)
\les \|\Lcalh w_h\|_{L^p_h(B_3)} \|I_h (\eta v_h)\|_{L^{p^\prime}(B_3)}
+ R_1^{-2}\|w_h\|_{W^{1,p}(B_3)} \|v_h\|_{L^{p^\prime}(B_3)}\\
&\qquad \nonumber 
+ \|w_h\|_{W^{2,p}_h(B_3)}\Bigl(\|\eta v_h-I_h (\eta v_h)\|_{L^{p^\prime}(B_3)}
+ h \|\nab (\eta v_h-I_h (\eta v_h))\|_{L^{p^\prime}(B_3)}\Bigr)\\
&\les \Bigl(\|\Lcalh w_h\|_{L^p_h(B_3)} 
+ \frac{1}{R_1^{2}} \|w_h\|_{W^{1,p}(B_3)} \Bigr)\|v_h\|_{L^{p^\prime}(B_3)},
\end{align*}
which implies that
\[
\|\Lcalh(\eta w_h)\|_{L^p_h(B_3)} \les \|\Lcalh w_h\|_{L^p_h(B_3)} 
+ \frac{1}{R_1^{2}} \|w_h\|_{W^{1,p}(B_3)}.
\]
Applying this upper bound to \eqref{LongLine1} yields 
\begin{align}\label{localestimateinproof}
\|w_h\|_{W^{2,p}_h(B_1)}  \les  \|\Lcalh w_h\|_{L^p_h(B_3)} 
+\frac{1}{R_1^{2}} \|w_h\|_{W^{1,p}(B_3)} \qquad \forall w_h\in V_h.
 \end{align}
 
\smallskip
{\em Step 2}: 
We now use a covering argument to obtain the global estimate \eqref{interiorstability}.
To this end, let $\{x_j\}_{j=1}^N\subset \Omega$ with 
$N=O(R_1^{-n})$ sufficiently large (but independent of $h$) such that 
$\overline{\Omega} = \cup_{j=1}^N \overline{B}_{{R_1}}(x_j)$.
Setting $S_j = B_{{R_1}}(x_j)$ and $\tilde{S}_j=B_{R_2}(x_j) = B_{2R_1}(x_j)$,
we have by \eqref{localestimateinproof} 
\begin{align*}
\|w_h\|_{W^{2,p}_h(\Omega)}^p
&\le \sum_{j=1}^N \|w_h\|^p_{W^{2,p}_h(S_j)}
\les \sum_{j=1}^N \Bigl(\|\Lcalh w_h\|^p_{L^p_h(\tilde{S}_j)} 
+ \frac{1}{R_1^{2p}} \|w_h\|^p_{W^{1,p}(\tilde{S}_j)} \Bigr)\\
&\les  \frac{1}{R_1^{2p}} \|w_h\|_{W^{1,p}(\Omega)}^p
+\sum_{j=1}^N \|\Lcalh w_h\|_{L^p_h(\tilde{S}_j)}^p.
\end{align*}
Since $V_h(\tilde{S}_j)\subseteq V_h$,
we have
\begin{align*}
\sum_{j=1}^N \|\Lcalh w_h\|_{L^p_h(\tilde{S}_j)}^p
&= \sum_{j=1}^N \bigg|\sup_{0\neq v_h\in V_h(\tilde{S}_j)}
\frac{\bigl(\Lcalh w_h,v_h \bigr)}{\|v_h\|_{L^{p^\prime}(\tilde{S}_j)}}\bigg|^p 
=\sum_{j=1}^N \bigg|\sup_{0\neq v_h\in V_h(\tilde{S}_j)}
\frac{\bigl( \Lcalh w_h,v_h \bigr)}{\|v_h\|_{L^{p^\prime}(\Omega)}}\bigg|^p\\
&\le N \bigg| \sup_{0\neq v_h\in V_h} 
\frac{ \bigl( \Lcalh w_h,v_h \bigr) }{\|v_h\|_{L^{p^\prime}(\Ome)}} \bigg|^p 
 \les \frac{1}{R_1^{n}} \|\Lcalh w_h\|_{L^p_h(\Ome)}^p.
 \end{align*}
Consequently, since $R_1$ is independent of $h$, we have
\begin{align*}
\|w_h\|_{W^{2,p}_h(\Omega)} &\les \frac{1}{R_1^{\frac{n}{p}}}  \|\Lcalh w_h\|_{L^p_h(\Omega)} 
+\frac{1}{R_1^{2}}\|w_h\|_{W^{1,p}(\Omega)}
\les \|\Lcalh w_h\|_{L^p_h(\Omega)}+ \|w_h\|_{W^{1,p}(\Omega)}.
\end{align*}

Finally, an application of Lemma \ref{DiscreteInterp} yields
\begin{align*}
\|w_h\|_{W^{2,p}_h(\Omega)}\les  \|\Lcalh w_h\|_{L^p_h(\Omega)} 
+  \|w_h\|^{\frac12}_{L^{p}(\Omega)}\|w_h\|^{\frac12}_{W^{2,p}_h(\Omega)}.
\end{align*}
Applying the Cauchy-Schwarz inequality to the last term completes the proof.
\end{proof}

Using arguments analogous to those in Lemma \ref{localStabilityLemma}, we also 
have the following stability estimate for the formal adjoint operator.
Due to its length and technical nature, we give the proof in the appendix.
\begin{lem}\label{DualGardingEstimateLemma}
There exists an $h_2>0$ such that
\begin{align}\label{DualEstimateLine}
\|v_h\|_{L^{p^\prime}(\Omega)}\les \sup_{0\neq w_h\in V_h} 
\frac{(\Lcalh w_h,v_h)}{\|w_h\|_{W_h^{2,p}(\Omega)}}\qquad \forall v_h\in V_h
\end{align}
provided $h\le h_*:=\min\{h_1,h_2\}$ and $k\ge 2$.
\end{lem}
\begin{remark}
Denote by $\Lcalh^*$ 
 the formal
adjoint operator of $\Lcalh$.  Then inequality \eqref{DualEstimateLine}
is  equivalent to the stability estimate
\begin{align}
\label{AdjointStability}
\|v_h\|_{L^{p^\prime}(\Omega)}
\les \sup_{0\neq w_h\in V_h} \frac{( \Lcalh^* v_h,w_h)}{\|w_h\|_{W_h^{2,p}(\Omega)}}\quad \forall v_h\in V_h.
\end{align}
Thus, the adjoint operator $\Lcalh^*$ is injective on $V_h$.  Since
$V_h$ is finite dimensional, $\Lcalh^*$ on $V_h$ is an isomorphism.
This implies that $\Lcalh$ is also an isomorphism on $V_h$; 
the stability of the operator is addressed in the next theorem, the main
result of this section.
\end{remark}

\begin{thm}\label{MainTheorem}
Suppose that  $h\le \min\{h_1,h_2\}$, and $k\ge 2$.  Then there holds the following stability estimate:
\begin{align}\label{MainStabilityEstimateALine}
\|w_h\|_{W^{2,p}_h(\Omega)}\les \|\Lcalh w_h\|_{L^p_h(\Omega)}\qquad \forall w_h\in V_h.
\end{align}
Consequently, there exists a unique solution to \eqref{MethodShort} satisfying
\begin{align}\label{MethodAPriori}
\|u_h\|_{W^{2,p}_h(\Omega)}\les \|f\|_{L^p(\Omega)}.
\end{align}
\end{thm}
\begin{proof}
For a given $w_h\in V_h$, Lemma \ref{DualGardingEstimateLemma} guarantees
the existence of a unique $\psi_h\in V_h$ satisfying 
 \begin{align}\label{AdjointProblem}
 ( \Lcalh v_h, \psi_h) = \int_\Omega w_h |w_h|^{p-2} v_h\, dx\qquad \forall v_h\in V_h.
 \end{align}
By \eqref{DualEstimateLine} we have
 \begin{align*}
 \|\psi_h\|_{L^{p^\prime}(\Omega)}\les \sup_{0\neq v_h\in V_h}
 \frac{( \Lcal_hv_h, \psi_h) }{\|v_h\|_{W^{2,p}_h(\Omega)}} 
& =  \sup_{0\neq v_h\in V_h}
 \frac{\int_\Omega w_h |w_h|^{p-2}  v_h\, dx }{\|v_h\|_{W^{2,p}_h(\Omega)}} 
 \les \|w_h\|_{L^p(\Omega)}^{p-1}.
\end{align*} 
The last inequality is an easy consequence of H\"older's inequality, Lemma
\ref{DiscreteInterp} and the Poincar\`e-Friedrichs inequality.
Taking $v_h = w_h$ in \eqref{AdjointProblem}, we have
\begin{align*}
\|w_h\|_{L^p(\Omega)}^p
&= ( \Lcalh w_h,\psi_h)\le \|\Lcalh w_h\|_{L^p_h(\Omega)}\|\psi_h\|_{L^{p^\prime}(\Omega)}
\le \|\Lcalh w_h\|_{L^p_h(\Omega)} \|w_h\|^{p-1}_{L^p(\Omega)},
\end{align*}
and therefore
\[
\|w_h\|_{L^p(\Omega)}\les \|\Lcalh w_h\|_{L^p_h(\Omega)}.
\]
Applying this estimate in \eqref{localStabilityLemma} proves \eqref{MainStabilityEstimateALine}.

Finally, to show existence and uniqueness of the finite element method 
\eqref{MethodShort} it suffices to show the estimate \eqref{MethodAPriori}.  
This immediately follows from \eqref{MainStabilityEstimateALine}
and H\"older's inequality:
\begin{align*}
\|u_h\|_{W^{2,p}_h(\Omega)} &\les \|\Lcalh u_h\|_{L^p_h(\Omega)} 
= \sup_{0\neq v_h\in V_h} \frac{( \Lcalh u_h,v_h)}{\|v_h\|_{L^{p^\prime}(\Omega)}} \\
&= \sup_{0\neq v_h\in V_h}\frac{\int_\Omega f v_h\, dx}{\|v_h\|_{L^{p^\prime}(\Omega)}}\le \|f\|_{L^p(\Omega)}.
\end{align*}\hfill
\end{proof}

\subsection{Convergence analysis}
The stability estimate in Theorem \ref{MainTheorem} immediately
gives us the following error estimate in the $W^{2,p}_h$ semi-norm.
\begin{thm}\label{ErrorEstimateThm1}
Assume that the hypotheses of Theorem {\rm\ref{MainTheorem}}
are satisfied.  Let $u\in W^{2,p}(\Omega)$ and $u_h\in V_h$ denote the solution to 
\eqref{problem} and \eqref{MethodShort}, respectively.  Then there holds
\begin{align}\label{MainThmLine1}
\|u-u_h\|_{W^{2,p}_h(\Omega)}\les \inf_{w_h\in V_h}\|u-w_h\|_{W^{2,p}_h(\Omega)}.
\end{align}
Consequently, if $u\in W^{s,p}(\Omega)$, for some $s\ge 2$, there holds
\begin{align*}
\|u-u_h\|_{W^{2,p}_h(\Omega)}\les h^{\ell-2}\|u\|_{W^{\ell,p}(\Omega)},
\end{align*}
where $\ell = \min\{s,k+1\}$.
\end{thm}

\begin{proof}
By  Theorem \ref{MainTheorem} and the consistency of the method, we have $\forall v_h\in V_h$
\begin{align*}
\|u_h-w_h\|_{W^{2,p}_h(\Omega)}
&\les \|\Lcalh (u_h-w_h)\|_{L^p_h(\Omega)}
= \sup_{0\neq v_h\in V_h} \frac{(\Lcalh (u_h-w_h),v_h)}{\|v_h\|_{L^{p^\prime}(\Omega)}} \\
&= \sup_{0\neq v_h\in V_h} \frac{a_h(u_h-w_h,v_h)}{\|v_h\|_{L^{p^\prime}(\Omega)}}
= \sup_{0\neq v_h\in V_h} \frac{a_h(u-w_h,v_h)}{\|v_h\|_{L^{p^\prime}(\Omega)}}\\
&= \sup_{0\neq v_h\in V_h} \frac{( \Lcalh (u-w_h),v_h)}{\|v_h\|_{L^{p^\prime}(\Omega)}}
\les  \|u-w_h\|_{W^{2,p}_h(\Omega)}.
\end{align*}
Applying the triangle inequality yields \eqref{MainThmLine1}.
\end{proof}

\section{Numerical experiments}\label{sec-4}
In this section we present several numerical 
experiments to show the efficacy of the finite element
method, as well as to validate the convergence theory.
In addition, we perform numerical experiments
where the coefficient matrix is not continuous and/or degenerate.
While these situations violate some of the assumptions given
in Section \ref{sec-3.1},
the tests show that the finite element method is effective for
these cases as well.

\subsection*{Test 1: H\"older continuous coefficients and smooth solution}
In this test we take $\Omega = (-0.5,0.5)^2$, the coefficient matrix to be
\begin{align*}
A = \begin{pmatrix}
|x|^{1/2}+1 & -|x|^{1/2}\\
-|x|^{1/2} & 5|x|^{1/2}+1
\end{pmatrix}
\end{align*}
and choose $f$ such that $u = \sin(2\pi x_1)\sin(\pi x_2)\exp(x_1\cos(x_2))$ as the exact solution.

The resulting $H^1$ and piecewise $H^2$ errors for various values
of polynomial degree $k$ and discretization parameter $h$are depicted
in Figure \ref{FigureTest1}.  The figure clearly indicates that 
the errors have the following behavior:
\begin{align*}
|u-u_h|_{H^1(\Omega)} = \mathcal{O}(h^k),\quad \|D^2_h(u-u_h)\|_{L^2(\Omega)} = \mathcal{O}(h^{k-1}).
\end{align*}
The second estimate is in agreement with Theorem \ref{ErrorEstimateThm1}.  In addition, 
the numerical experiments suggest that (i) the method converges with optimal order in the $H^1$-norm
and (ii) the method is convergent in the piecewise linear case ($k=1$).

\begin{figure}[htb] 
\centerline{
   \includegraphics[height=2.3in,width=2.5in]{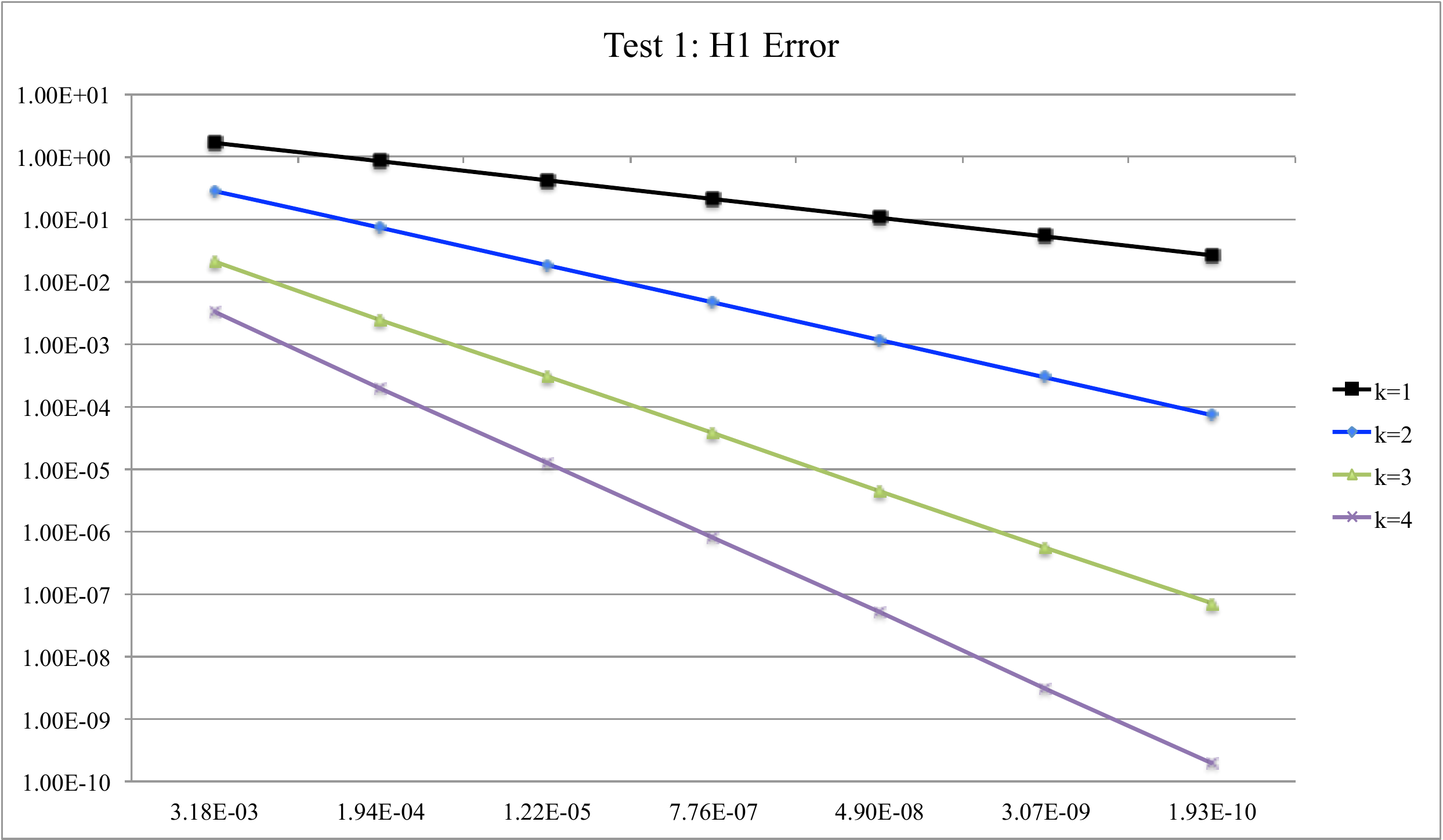} \,\, 
      \includegraphics[height=2.3in,width=2.5in]{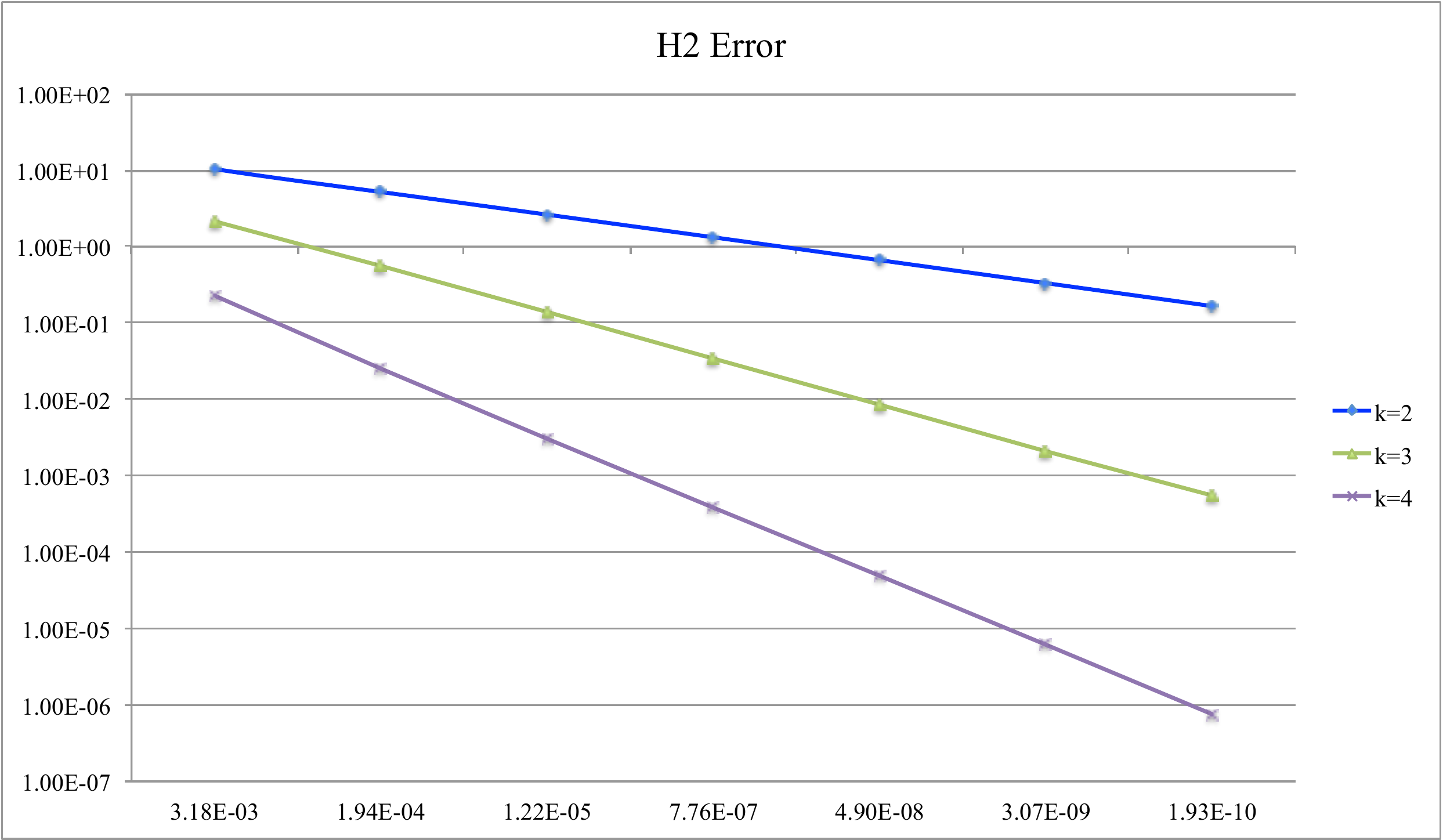} 
}
   \caption{The $H^1$ (left) and piecewise $H^2$ (right) errors for Test Problem 1
   with polynomial degree $k=1,2,3,4$.  The figures show that the $H^1$ error
   converges with order $\mathcal{O}(h^k)$, where as the piecewise $H^2$ error converges with
   order $\mathcal{O}(h^{k-1})$.}
\label{FigureTest1}
\end{figure}

\subsection*{Test 2: Uniformly continuous coefficients and $W^{2,p}$ solution}
For the second set of numerical experiments, we take the domain 
to be the  square $\Omega =(0,1/2)^2$, and take the coefficient matrix to be
\begin{align*}
A = 
\begin{pmatrix}
-\dfrac{5}{\log(|x|)}+15 & 1\\
1 & - \dfrac{1}{\log(|x|)} +3
\end{pmatrix},
\end{align*}
We choose the data such that the exact solution is given by $u = |x|^{7/4}$.
We note that $u\in W^{m,p}(\Omega)$ for $(7-4m)p>-8$.  In particular, $u\in W^{2,p}(\Omega)$
for $p<8$ and $u\in W^{3,p}(\Omega)$ for $p<8/5$.

In order to apply Theorem \ref{ErrorEstimateThm1} to this test problem, we
recall that the $k$th degree nodal interpolant of $u$ with $k\ge 2$ satisfies
\begin{align*}
\|D^2_h(u-I_h u)\|_{L^2(\Omega)}\le C h^{2-2/p}\|u\|_{W^{3,p}(\Omega)}
\end{align*}
for $p<2$.  Since $u\in W^{3,p}(\Omega)$ for $p<8/5$, Theorem \ref{ErrorEstimateThm1}
then predicts the convergence rate 
\begin{align*}
\|D^2_h(u-u_h)\|_{L^2(\Omega)}\le C\|D^2_h(u-I_h u)\|_{L^2(\Omega)}=\mathcal{O}(h^{3/4-\eps})
\end{align*}
for any $\eps>0$.  Note that a slight modification of these arguments also shows
that $|u-I_h u|_{H^1(\Omega)} = \mathcal{O}(h^{7/4-\eps})$.

The errors of the finite element method for {Test 2} 
using piecewise linear, quadratic and cubic polynomials are depicted
in Figure \ref{FigureTest2}.  As predicted by the theory, the $H^2$
error converges with order $\approx \mathcal{O}(h^{3/4})$ if the polynomial
degree is greater than or equal to two.  Similar to the first
test problem, the numerical experiments also show that the $H^1$ error
converges with optimal order, i.e., $|u-u_h|_{H^1(\Omega)} = \mathcal{O}(h^{7/4-\eps})$.

\begin{figure}[htb] 
 \centerline{ 
   \includegraphics[height=2.3in,width=2.5in]{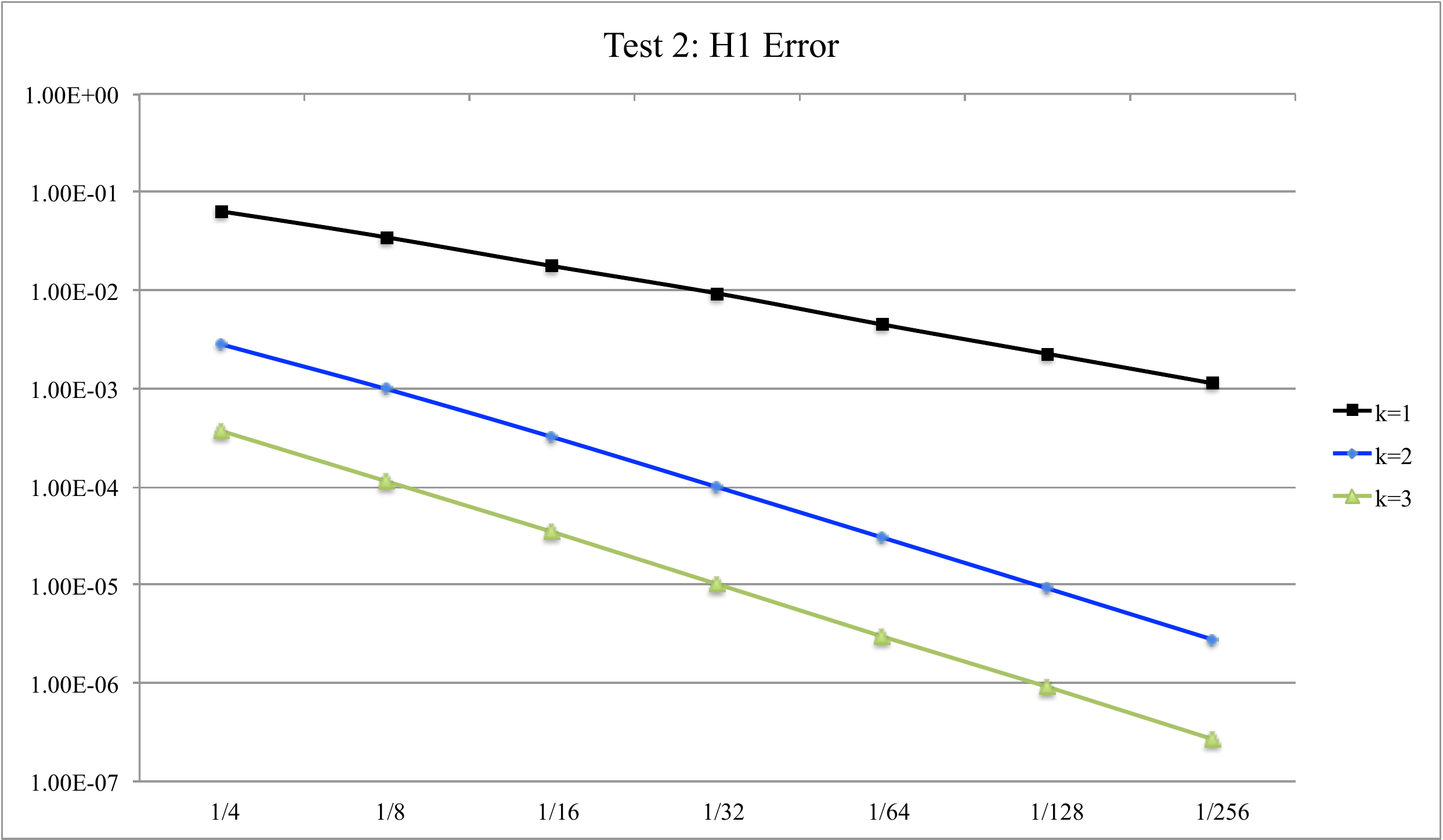}\,\,
      \includegraphics[height=2.3in,width=2.5in]{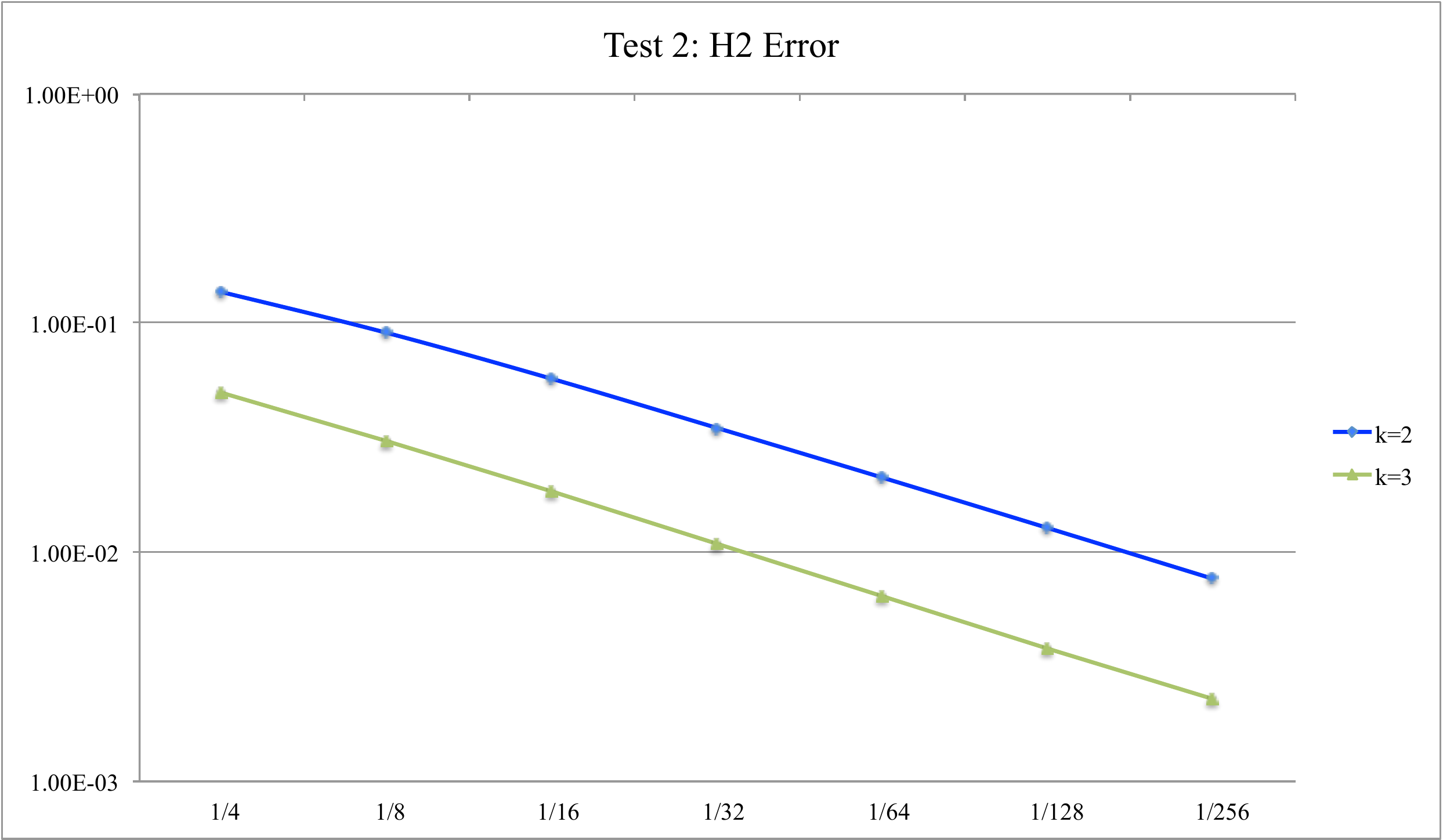} 
}
   \caption{The $H^1$ (left) and piecewise $H^2$ (right) errors for Test Problem 2
   with polynomial degree $k=1,2,3$.  The figures show that the $H^1$ error
   converges with order $\mathcal{O}(h^{\min\{k,7/4-\eps\}})$, where as the piecewise 
   $H^2$ error converges with order $\mathcal{O}(h^{\min\{k,7/4-\eps\}-1})$.}
\label{FigureTest2}
\end{figure}

\subsection*{Test 3: Degenerate coefficients and $W^{2,p}$ solution}
For the third and final set of test problems, we take $\Omega =(0,1)^2$, 
\begin{align*}
A = 
\frac{16}{9} 
\begin{pmatrix}
x_1^{2/3} & - x_1^{1/3}x_2^{1/3}\\
-x_1^{1/3} x_2^{1/3} & x_2^{2/3}
\end{pmatrix},
\end{align*}
and exact solution $u = x_1^{4/3}-x_2^{4/3}$.
We remark that the choice of the matrix and solution
is motivated by Aronson's example for the infinity-Laplace equation.
In particular, the function $u$ satisfies the quasi-linear PDE $\Del_\infty u = 0$, where
$\Del_\infty u:=(D^2 u \nab u)\cdot \nab u = (D^2 u):(\nab u (\nab u)^T)$.
Noting that $A = \nab u (\nab u)^T$, we see that $-A:D^2 u=0=:f$.

Unlike the first two test problems, the matrix is not uniformly elliptic, as
$\det(A(x)) =0$ for all $x\in \Omega$.  Therefore the theory given 
in the previous sections does not apply.  We also note that the exact
solution satisfies the regularity $u\in W^{m,p}(\Omega)$ for $(4-3m)p>-1$,
and therefore $u\in W^{2,p}(\Omega)\cap W^{1,\infty}(\Omega)$ for $p<3/2$.

The resulting errors of the finite element method using piecewise linear
and quadratic polynomials are plotted in Figure \ref{FigureTest3}.
In addition, we plot the computed solution and error in Figure \ref{FigureTestPicture}
with $k=2$ and $h=1/256$.
While this problem is outside the scope of the theory, the experiments
show that the method converges, and the following rates are observed:
\begin{align*}
\|u-u_h\|_{L^2(\Omega)} = \mathcal{O}(h^{4/3}),\qquad |u-u_h|_{H^1(\Omega)} = \mathcal{O}(h^{5/6}).
\end{align*}

\begin{figure}[htb] 
\centerline{
   \includegraphics[height=2.3in,width=2.5in]{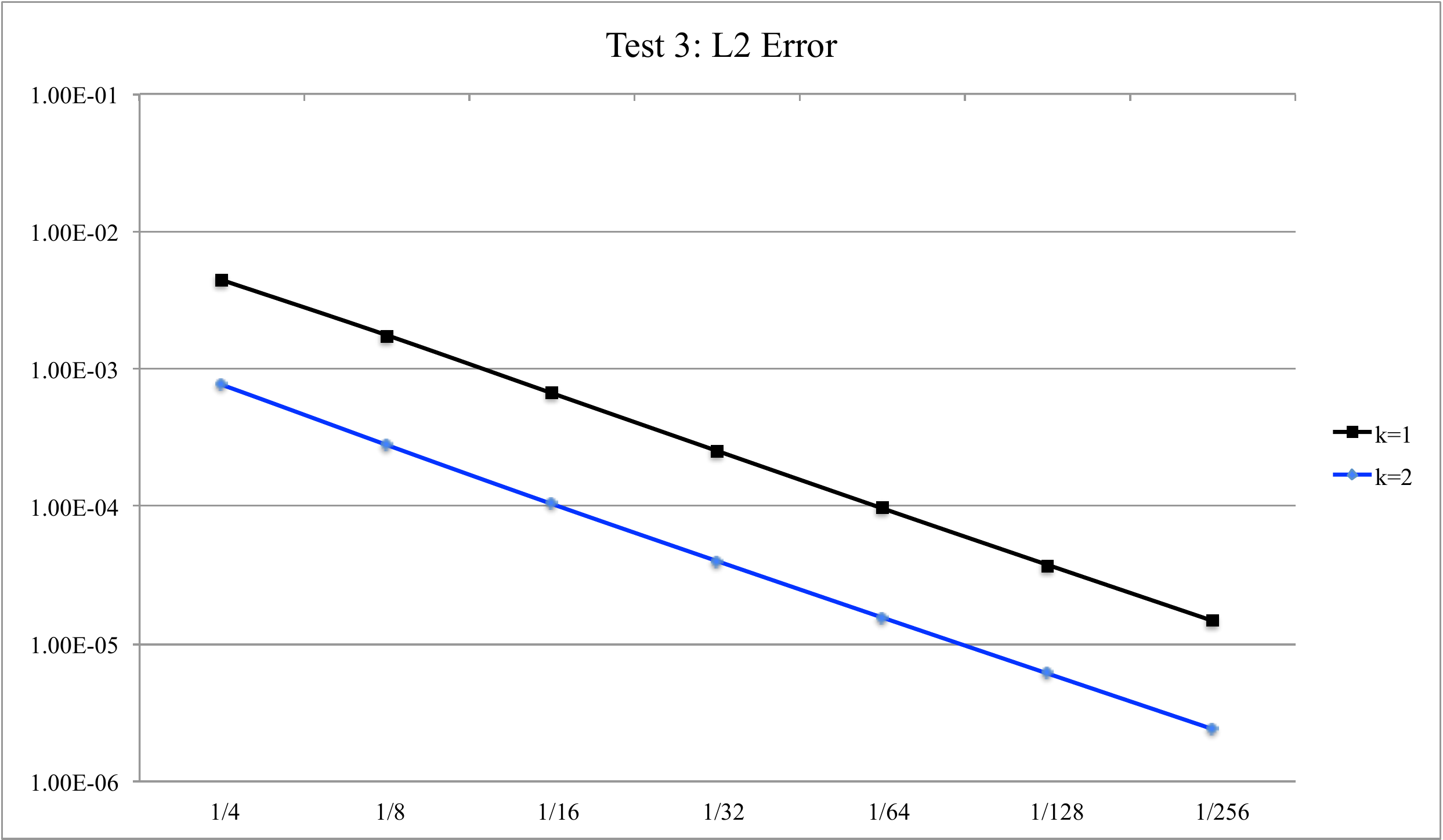} \,\,\,
      \includegraphics[height=2.3in,width=2.5in]{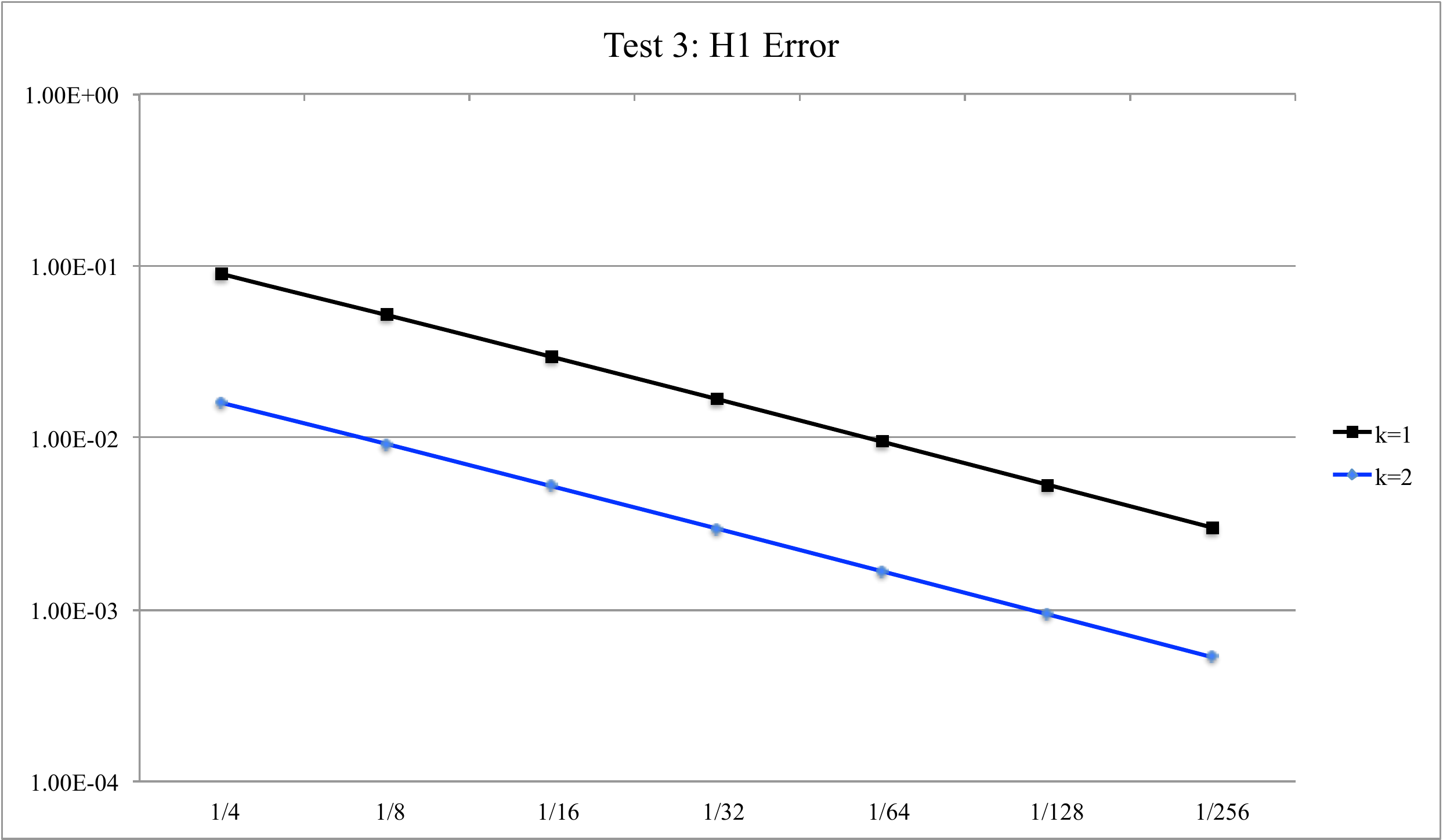} 
}
   \caption{The $H^1$ (left) and piecewise $H^2$ (right) errors for the degenerate Test Problem 3
   with polynomial degree $k=1$ and $k=2$.  The figures show that the $L^2$ error
   converges with order $\approx \mathcal{O}(h^{4/3})$ and the $H^1$ error converges
   with order $\approx \mathcal{O}(h^{5/6})$.}
\label{FigureTest3}
\end{figure}

\begin{figure}[htb] 
   \centering
   \includegraphics[height=2.3in,width=2.5in]{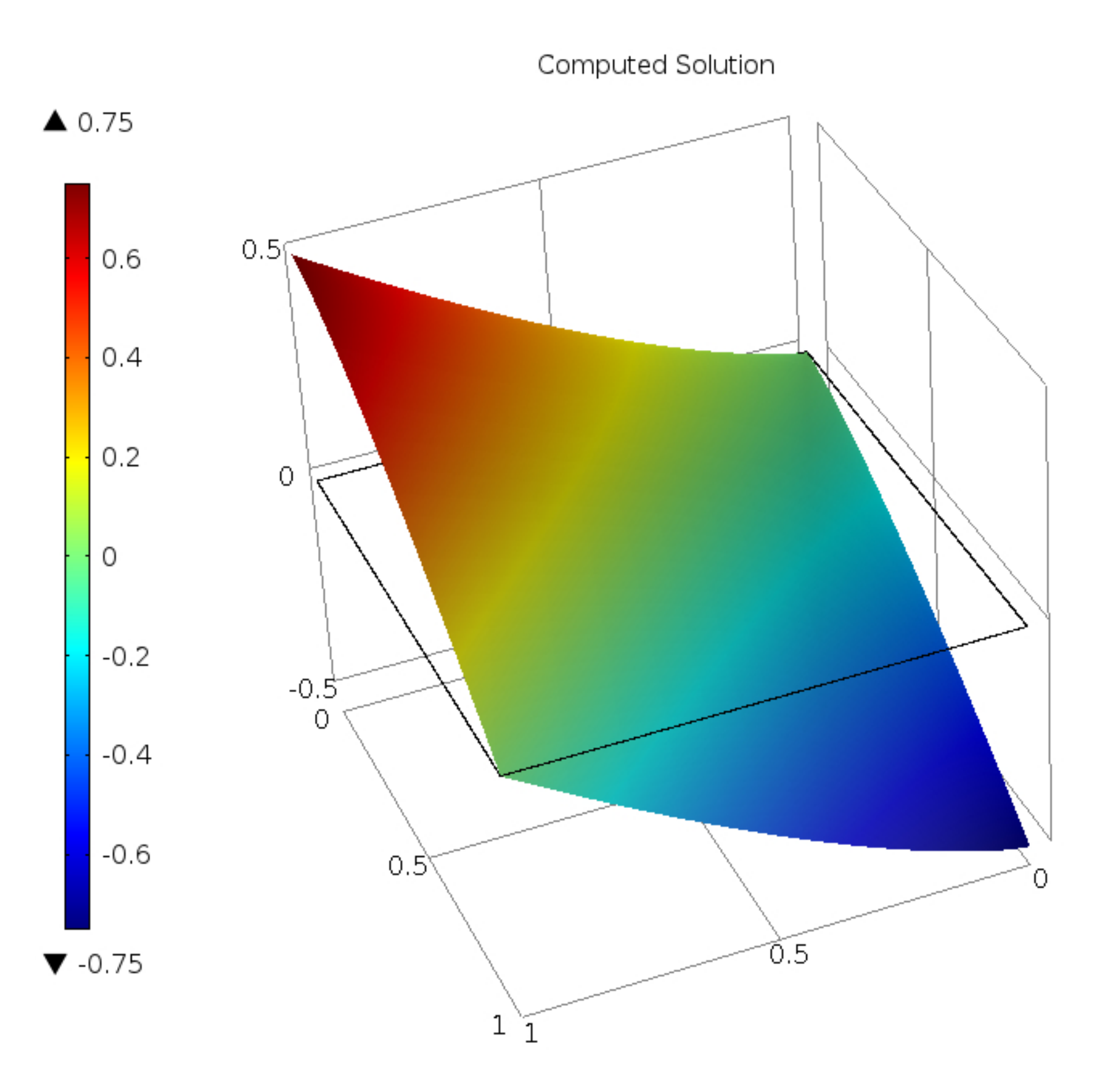}\,\,
      \includegraphics[height=2.3in,width=2.5in]{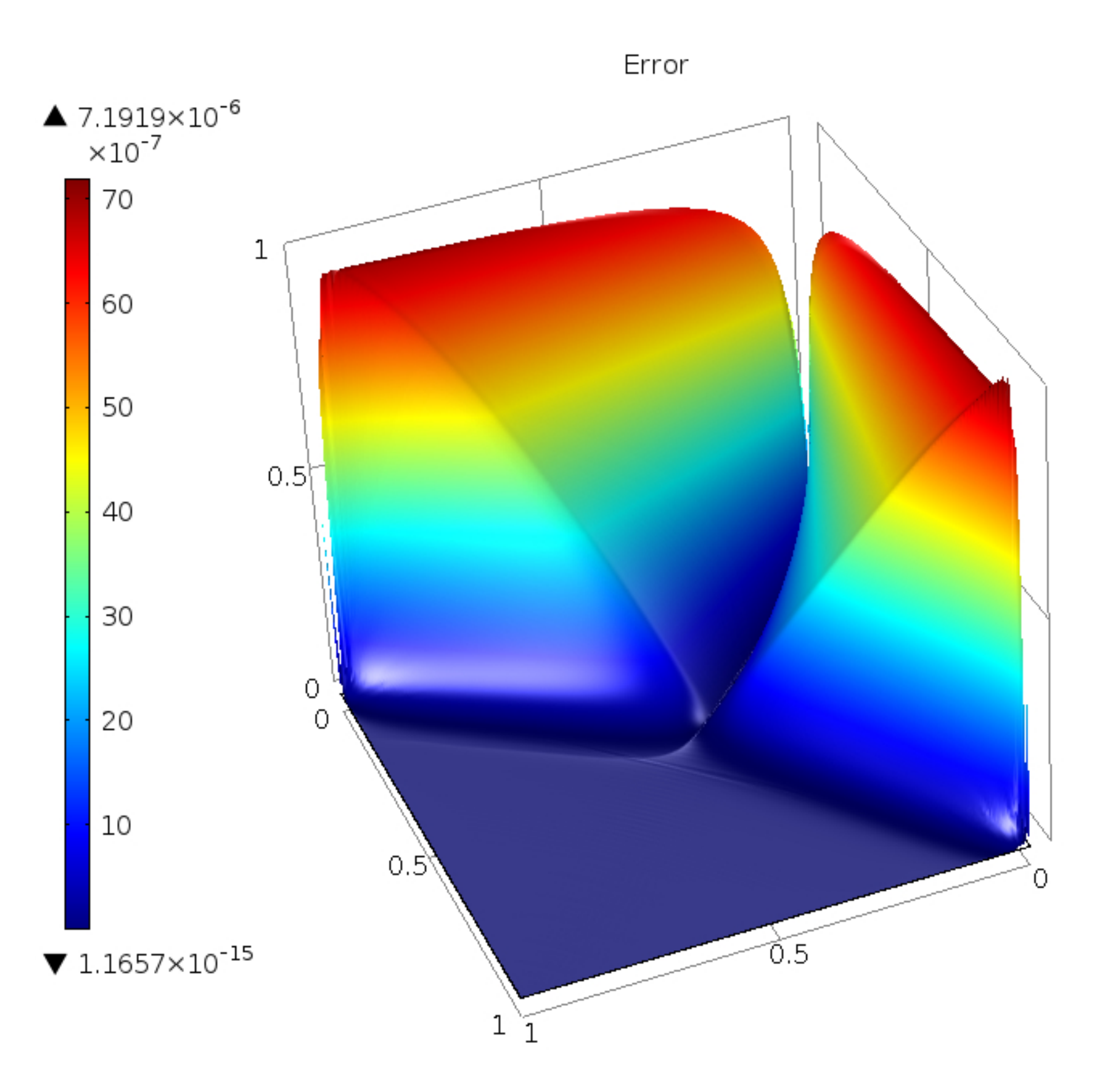} 
   \caption{Computed solution (left) and error (right) of test problem 3 with $k=2$ and $h=1/256$.}
\label{FigureTestPicture}
\end{figure}

{

}

\appendix

\section{Super approximation result}\label{AppendixA}
Here, we provide the proof of Lemma \ref{Superlem}.
As a first step, we use standard interpolation estimates \cite{Ciarlet78,Brenner} to obtain for $0\le m\le k+1$,
\begin{align}\label{standardInterpEst}
h^{mp} \|\eta v_h-I_h (\eta v_h)\|^p_{W^{m,p}(T)}
\les h^{p(k+1)} |\eta v_h|^p_{W^{k+1,p}(T)}.
\end{align}
Since $|\eta|_{W^{j,\infty}(T)}\les d^{-j}$ and $|v_h|_{H^{k+1}(T)} =0$, we find
\begin{align}\label{SuperProofLine1}
|\eta v_h|_{W^{k+1,p}(T)} 
&\les \sum_{|\alpha|+|\beta|=k+1} \int_T |D^\alpha \eta|^p |D^\beta v_h|^p\, dx\\
&\nonum\les \sum_{j=0}^{k} \frac{1}{d^{p(k+1-j)}} |v_h|_{W^{j,p}(T)}^p
\les \sum_{j=0}^k \frac{h^{-jp}}{d^{p(k+1-j)}} \|v_h\|_{L^p(T)}^p,
\end{align}
where an inverse estimate was applied to derive the last inequality.
Combining \eqref{SuperProofLine1} with \eqref{standardInterpEst} and using
the hypothesis $h\le d$ then gives us
\begin{align*}
h^{mp}\|\eta v_h-I_h (\eta v_h)\|_{W^{m,p}(T)}^p \les \sum_{j=0}^k \frac{h^{p(k+1-j)}}{d^{p(k+1-j)}} \|v_h\|_{L^p(T)}^p
\les \frac{h^{p}}{d^p}\|v_h\|_{L^p(T)}^p.
\end{align*}
Therefore for $m\in \{0,1\}$ we have
\begin{align*}
h^{mp}  \|\eta v_h-I_h (\eta v_h)\|^p_{W^{m,p}(D)}
&\le \mathop{\sum_{T\in \mct}}_{T\cap D\neq \emptyset} h^{mp}\|\eta v_h-I_h (\eta v_h)\|^p_{W^{m,p}(T)}\\
&\les \mathop{\sum_{T\in \mct}}_{T\cap D\neq \emptyset} \frac{h^p}{d^p} \|v_h\|_{L^p(T)}^p
\le \frac{h^p}{d^p}\|v_h\|_{L^p(D_h)}^p.
\end{align*}
Thus, \eqref{SuperLine2} is satisfied.

To obtain the second estimate \eqref{SuperLine1}, we first use \eqref{standardInterpEst},
\eqref{SuperProofLine1} an an inverse estimate to get
\begin{align}\label{SuperProofLine2}
\|D^2 (\eta v_h-I_h (\eta v_h))\|_{L^p(T)}^p
&\les h^{p(k-1)}|\eta v_h|_{W^{k+1,p}(T)}^p\les \sum_{j=0}^k \frac{h^{p(k-1)}}{d^{p(k+1-j)}} |v_h|_{W^{j,p}(T)}^p\\
&\nonum\les \frac{1}{d^{2p}} \|v_h\|_{L^p(T)}^p + \sum_{j=1}^k \frac{h^{k-j}}{d^{p(k+1-j)}} \|v_h\|_{W^{1,p}(T)}^p\\
&\nonum\les \frac{1}{d^{2p}} \|v_h\|_{L^p(T)}^p + \frac{1}{d^{p}} \|v_h\|_{W^{1,p}(T)}^p\les \frac{1}{d^{2p}} \|v_h\|_{W^{1,p}(T)}^p.
\end{align}
By similar arguments we find
\begin{align}
\label{SuperProofLine3}
h^{-p}\|\nab (\eta v_h-I_h (\eta v_h))\|_{L^p(T)}^p
\les h^{p(k-1)}|\eta v_h|_{W^{k+1,p}(T)}^p\les \frac{1}{d^{2p}} \|v_h\|_{W^{1,p}(T)}^p.
\end{align}
Therefore by Lemma \ref{TraceLine} and \eqref{SuperProofLine2}--\eqref{SuperProofLine3}, we obtain
\begin{align*}
\|\eta v_h-I_h (\eta v_h)\|_{W^{2,p}_h(D)}^p
&\le \mathop{\sum_{T\in \mct}}_{T\cap D\neq \emptyset}\|D^2(\eta v_h-I_h (\eta v_h))\|_{L^p(T)}^p\\
&\qquad 
+ {\sum_{e\in \mce^I}} h_e^{1-p}\big\|\jump{\nab (\eta v_h-I_h (\eta v_h))}\big\|_{L^p(e\cap \bar{D})}^p\\
&\les \mathop{\sum_{T\in \mct}}_{T\cap D\neq \emptyset}\Big[\|D^2(\eta v_h-I_h (\eta v_h))\|_{L^p(T)}^p
+h^{-p}\|\nab (\eta v_h-I_h (\eta v_h))\big\|_{L^p(T)}^p\Big]\\
&\les \mathop{\sum_{T\in \mct}}_{T\cap D\neq \emptyset} \frac{1}{d^{2p}} \|v_h\|_{W^{1,p}(T)}^p\le \frac{1}{d^{2p}} \|v_h\|_{W^{1,p}(D_h)}^p.
\end{align*}
Taking the $p$th root of this last expression yields the estimate \eqref{SuperLine1}.
The proof of \eqref{SuperLine3} uses the exact same arguments and is therefore omitted.

\section{Proof of Lemma \ref{DualGardingEstimateLemma}}
To prove Lemma \ref{DualGardingEstimateLemma} we introduce
the discrete $W^{-2,p}$-type norm
\begin{align}\label{Hm2norm}
\|r\|_{W^{-2,p}_h(D)}:=\sup_{0\neq v_h\in V_h(D)} \frac{(r,v_h)_D}{\|v_h\|_{W^{2,p^\prime}(D)}},
\end{align}
and the $W^{-1,p}$ norm (defined for $L^p$ functions)
\begin{align}\label{Hm1norm}
\|r\|_{W^{-1,p}(D)} = \sup_{0\neq v\in W^{1,p^\prime}(D)} \frac{(r,v)_D}{\|v\|_{W^{1,p^\prime}(D)}} = \mathop{\sup_{v\in W^{1,p^\prime}}}_{\|v\|_{W^{1,p^\prime}(D)}=1} 
(r,v)_D\, dx.
\end{align}
The desired estimate \eqref{DualEstimateLine} is then equivalent to 
\begin{align}
\label{AdjointGoal}
\|v_h\|_{L^{p^\prime}(\Omega)}\les \|\Lcalh^* v_h\|_{W^{-2,p^\prime}_h(\Omega)}\qquad \forall v_h\in V_h,
\end{align}
where we recall that $\Lcalh^*$ is the adjoint operator of $\Lcalh$.
Due to its length, we break up the proof of \eqref{AdjointGoal} 
into three steps.\medskip

\noindent {\em Step 1:  A local estimate.}\\
\indent The first step in the derivation of \eqref{DualEstimateLine} (equivalently, \eqref{AdjointGoal}) is to prove
a local version of this estimate, analogous to Lemma \ref{LocalNonDivLemma}.
To this end, for fixed $\xn\in \Omega$, let $\delta_0$, $R_{\delta_0}$, $R_1:=\frac13 R_{\delta_0}$
and $B_1 :=B_{R_1}(x_0)$ 
be as in Lemmas \ref{operatorDiffForm}--\ref{LocalNonDivLemma},
with $\delta_0>0$ to be determined.
For a fixed  $v_h\in V_h(B_{1})$, let $\varphi\in W^{2,p}(\Omega)\cap W^{1,p}_0(\Omega)$ satisfy
$\Lcal \varphi = v_h|v_h|^{p^\prime-2}$ in $\Omega$ with 
\begin{align}
\label{AppEllipticReg}
\|\varphi\|_{W^{2,p}(\Omega)}\les \||v_h|^{p^\prime-1}\|_{L^p(\Omega)} \les \|v_h\|_{L^{p^\prime}(B_{1})}^{p^\prime-1}.
\end{align}
Multiplying the PDE by $v_h$, integrating over $\Omega$,
 and using the consistency of $\Lcalh$ yields
\begin{align*}
\|v_h\|^{p^\prime}_{L^{p^\prime}(B_{1})} = \|v_h\|^{p^\prime}_{L^{p^\prime}(\Omega)} = (\Lcal \varphi,v_h) = ( \Lcalh \varphi,v_h).
\end{align*}
Therefore, for any $\varphi_h\in V_h$, there holds
\begin{align}\label{LineInLemma}
\|v_h\|_{L^{p^\prime}(B_1)}^{p^\prime} 
&= ( \Lcalh \varphi_h,v_h) + (\Lcalh (\varphi-\varphi_h),v_h)\\
&\nonum= ( \Lcalh^* v_h, \varphi_h) + ( \Lcaloh (\varphi-\varphi_h),v_h)+ ( (\Lcalh-\Lcaloh)(\varphi-\varphi_h),v_h),
\end{align}
where $\Lcaloh$ is given by \eqref{add9c} with $A_0 \equiv A(x_0)$.
Now take $\varphi_h\in V_h$ to be the elliptic projection of $\varphi$ 
with respect to $\Lcaloh$, i.e.,
\begin{align*}
( \Lcaloh (\varphi-\varphi_h),w_h)=0\qquad \forall w_h\in V_h.
\end{align*}
Lemma \ref{StabilityLem} ensures that $\varphi_h$ is well--defined and satisfies the estimate
\begin{align}\label{DAppEllipticReg}
\|\varphi_h\|_{W^{2,p}_h(\Omega)}\les \|\Lcaloh \varphi_h\|_{L^p_h(\Omega)} = \|\Lcaloh \varphi\|_{L^p_h(\Omega)}\les \|\varphi\|_{W^{2,p}(\Omega)}\les \|v_h\|_{L^{p^\prime}(B_1)}^{p^\prime-1}.
\end{align}
Combining  Lemma  \ref{operatorDiffForm}, \eqref{AppEllipticReg}--\eqref{DAppEllipticReg} and \eqref{Hm2norm},
we have
\begin{align*}
&\|v_h\|_{L^{p^\prime}(B_{1})}^{p^\prime} 
= ( \Lcalh^* v_h,\varphi_h) + \big( (\Lcalh - \Lcaloh)(\varphi-\varphi_h),v_h\big)\\
&\quad\nonum\le \|\Lcalh^* v_h\|_{W^{-2,p^\prime}_h(\Omega)} \|\varphi_h\|_{W^{2,p}_h(\Omega)} + \| (\Lcalh - \Lcaloh)(\varphi-\varphi_h)\|_{L^p_h(B_{1})} \|v_h\|_{L^{p^\prime}(B_{1})}\\
&\quad\nonum\le \|\Lcalh^* v_h\|_{W^{-2,p^\prime}_h(B_{1})} \|v_h\|_{L^{p^\prime}(B_1)}^{p^\prime-1} + \delta_0\|\varphi-\varphi_h\|_{W^{2,p}_h(B_1)} \|v_h\|_{L^{p^\prime}(B_1)}\\
&\quad\nonum\le \|\Lcalh^* v_h\|_{W^{-2,p^\prime}_h(B_{1})} \|v_h\|_{L^{p^\prime}(B_1)}^{p^\prime-1} + \delta_0 \|v_h\|^{p^\prime}_{L^{p^\prime}(B_1)}.
\end{align*}
Taking $\delta_0$ sufficiently small and rearranging terms 
 gives the local stability estimate for finite element functions with compact support:
\begin{align}\label{AppStep1Line}
\|v_h\|_{L^{p^\prime}(B_1)}\les \|\Lcalh^* v_h\|_{W^{-2,p^\prime}_h(B_{1})}\qquad \forall v_h\in V_h(B_1).
\end{align}

\noindent {\em Step 2:  A global G\"arding-type inequality.}\\
We now follow the proof of Lemma \ref{localStabilityLemma} to derive 
a global G\"arding-type inequality for the adjoint problem.  Let $R_1$ be given in the first
step of the proof, $R_2 = 2R_1$, and $R_3 = 3R_1$.
Let  $\eta\in C^3(\Omega)$
satisfy the conditions in Lemma \ref{localStabilityLemma} (cf. \eqref{etaprop}).
By the triangle inequality and \eqref{AppStep1Line}
we have for any $v_h\in V_h$
\begin{align*}
&\|v_h\|_{L^{p^\prime}(B_1)} 
= \|\eta v_h\|_{L^{p^\prime}(B_1)} \le \|\eta v_h-I_h (\eta v_h)\|_{L^{p^\prime}(B_1)}+\|I_h (\eta v_h)\|_{L^{p^\prime}(B_1)}\\
&\quad \les \|\eta v_h-I_h (\eta v_h)\|_{L^{p^\prime}(B_1)}+ \|\Lcalh^* (I_h (\eta v_h))\|_{W^{-2,p^\prime}_h(B_{1})}\\
&\quad\les \|\eta v_h-I_h (\eta v_h)\|_{L^{p^\prime}(B_1)}+ \|\Lcalh^* (I_h (\eta v_h)-\eta v_h)\|_{W^{-2,p^\prime}_h(B_{1})}+ \|\Lcalh^* (\eta v_h)\|_{W^{-2,p^\prime}_h(B_{1})}.
\end{align*}
Applying  Lemmas \ref{LContLemma}, Lemma \ref{Superlem} (with $d=R_1$) and
an inverse estimate yields 
\begin{align}\label{AppBoundLine42}
\|v_h\|_{L^{p^\prime}(B_1)} 
&\les \|\eta v_h-I_h (\eta v_h)\|_{L^{p^\prime}(B_{2})}+ \|\Lcalh^* (\eta v_h)\|_{W^{-2,p^\prime}_h(B_{1})}\\
&\nonum \les \frac{h}{R_1} \|v_h\|_{L^{p^\prime}(B_{3})}+\|\Lcalh^* (\eta v_h)\|_{W^{-2,p^\prime}_h(B_{3})}\\
&\nonumber\les \frac{1}{R_1} \|v_h\|_{W^{-1,p^\prime}(B_3)} +\|\Lcalh^* (\eta v_h)\|_{W^{-2,p^\prime}_h(B_3)}.
\end{align}

The goal now is to replace $\Lcalh^*(\eta v_h)$
appearing in the right--hand side of \eqref{AppBoundLine42}
by $\Lcalh^* v_h$ plus low--order terms.
To this end, we write  for $w_h\in V_h(B_3)$ (cf. \ref{add9c}),
\begin{align}\label{IsLine}
&(\Lcalh^* (\eta v_h),w_h)
 = a_h(w_h,\eta v_h)
 = a_h(w_h \eta,v_h)+\big[a_h(w_h,\eta v_h)-a_h(w_h\eta ,v_h)\big]\\
&\quad\nonumber= a_h(I_h(w_h \eta),v_h)+a_h(w_h\eta -I_h(w_h\eta),v_h)+
\big[a_h(w_h,\eta v_h)-a_h(w_h\eta ,v_h)\big]\\
%
&\quad\nonum =: I_1+I_2+I_3.
\end{align}
To derive an upper bound of $I_1$, we  use \eqref{Hm2norm} 
and  properties of the interpolant and cut--off function $\eta$:
\begin{align}\label{AppI1Bound}
I_1 & = (\Lcalh^*v_h,I_h(\eta w_h)) \le \|\Lcalh^* v_h\|_{W^{-2,p^\prime}(B_3)}\|I_h (\eta w_h)\|_{W^{2,p}(B_3)}\\
&\nonum\les \|\Lcalh^* v_h\|_{W^{-2,p^\prime}(B_3)}\|\eta w_h\|_{W^{2,p}_h(B_3)}
\les \frac{1}{R_1^2}  \|\Lcalh^* v_h\|_{W^{-2,p^\prime}(B_2)}\|w_h\|_{W^{2,p}_h(B_3)}.
\end{align}
Next, we apply Lemmas \ref{LContLemma}, \ref{Superlem} and an inverse estimate 
to bound $I_2$:
\begin{align}\label{AppI2Bound}
I_2 
&= ( \Lcalh (\eta w_h-I_h (\eta w_h)),v_h)
 \les \|\eta w_h-I_h (\eta w_h)\|_{W^{2,p}_h(B_3)} \|v_h\|_{L^{p^\prime}(B_3)}\\
 &\nonum\les \frac{h}{R_1^3}\|w_h\|_{W^{2,p}_h(B_3)} \|v_h\|_{L^{p^\prime}(B_3)}
\nonum \les \frac{1}{R_1^3}\|w_h\|_{W^{2,p}_h(B_3)} \|v_h\|_{W^{-1,p^\prime}(B_3)}.
\end{align}

To estimate $I_3$, we add and subtract $a_0(w_h,\eta v_h)-a_0(w_h\eta,v_h)$
and expand terms to obtain
\begin{align}
I_3 \label{I3expand}
&= a_0(w_h,\eta v_h)-a_0(w_h \eta,v_h)\\
&\qquad\nonumber+\big[a_h(w_h,\eta v_h)-a_h(w_h \eta,v_h)-\big(a_0(w_h,\eta v_h)-a_0(w_h\eta,v_h)\big)\big]\\
&\nonumber = -\int_{B_3} \Big(w_h A_0:D^2 \eta +2 A_0\nab \eta \cdot \nab w_h\Big) v_h\, dx\\
&\nonumber  \qquad -\int_{B_3} \Big(w_h (A-A_0):D^2 \eta +2 (A-A_0)\nab \eta \cdot \nab w_h\Big) v_h\, dx =:K_1+K_2.
\end{align}
Applying H\"older's inequality 
and Lemmas \ref{discretePoincare}--\ref{discreteHolder} yields
\begin{align}\label{K1Bound}
K_1&\le \|w_h A_0:D^2\eta \|_{W^{1,p}(B_3)}\|v_h\|_{W^{-1,p^\prime}(B_3)}
+2\Big|\int_{B_3} (A_0 \nab \eta \cdot \nab w_h)v_h\, dx\Big|\\
&\nonumber\les \big(\frac{1}{R_1^3} \|w_h\|_{W^{1,p}(B_3)} + \frac{1}{R_1^2} \|w_h\|_{W^{2,p}_h(B_3)}\big)\|v_h\|_{W^{-1,p^\prime}(B_3)}\\
&\nonumber\les \frac{1}{R_1^2} \|w_h\|_{W^{2,p}_h(B_3)}
\|v_h\|_{W^{-1,p^\prime}(B_3)},
%
\end{align}
Similarly, by Lemma \ref{discretePoincare}
and \eqref{AC101}, we obtain
\begin{align}
\label{K2bound}
K_2
&\le \|A-A_0\|_{L^\infty(B_3)}\big(\|w_h\|_{L^p(B_3)} \|D^2\eta \|_{L^\infty(\Omega)}
+ \|\nab w_h\|_{L^p(B_3)}\|\nab \eta\|_{L^\infty(\Omega)}\big)\|v_h\|_{L^{p^\prime}(B_3)}\\
&\nonumber \les \delta_0 \big({R_3^2}\|D^2\eta \|_{L^\infty(\Omega)} +
R_3 \|\nab \eta\|_{L^\infty(\Omega)}\big)\|w_h\|_{W^{2,p}_h(B_3)}\|v_h\|_{L^{p^\prime}(B_3)}\\
&\nonumber \les  \delta_0\|w_h\|_{W^{2,p}_h(B_3)}\|v_h\|_{L^{p^\prime}(B_3)}
\end{align}
Combining \eqref{I3expand}--\eqref{K2bound}
results in the following upper bound of $I_3$:
\begin{align}\label{I3bound}
I_3 \les \frac{1}{R_1^2} \|w_h\|_{W^{2,p}_h(B_3)} \|v_h\|_{W^{-1,p^\prime}(B_3)}
+ \delta_0\|w_h\|_{W^{2,p}_h(B_3)}\|v_h\|_{L^{p^\prime}(B_3)}
\end{align}

Applying the estimates to \eqref{AppI1Bound}--\eqref{AppI2Bound}, \eqref{I3bound} to \eqref{IsLine} results in
\begin{align*}
( \Lcalh^* (\eta v_h),w_h) 
&\les \frac{1}{R_1^3}\Big( \|\Lcalh^* v_h\|_{W^{-2,p^\prime}(B_3)}+ \|v_h\|_{W^{-1,p^\prime}(B_3)}\Big)\|w_h\|_{W^{2,p}_h(B_3)}\\
&\qquad +\delta_0 \|v_h\|_{L^{p^\prime}(B_3)}\|w_h\|_{W^{2,p}_h(B_3)}.
\end{align*}
and therefore by \eqref{Hm2norm},
\begin{align}
\label{AppBoundLine57}
\|\Lcalh^* (\eta v_h)\|_{W^{-2,p^\prime}_h(B_3)}\les \frac{1}{R_1^3} \Big(\|\Lcalh^* v_h\|_{W^{-2,p^\prime}(B_3)}+ \|v_h\|_{W^{-1,p^\prime}(B_3)}\Big)+\delta_0\|v_h\|_{L^{p^\prime}(B_3)}.
\end{align}
Combining \eqref{AppBoundLine57} and \eqref{AppBoundLine42} yields
\begin{align*}
\|v_h\|_{L^{p^\prime}(B_1)} \les  \frac{1}{R_1^3} \Big(\|\Lcalh^* v_h\|_{W^{-2,p^\prime}_h(B_3)}+ \|v_h\|_{W^{-1,p^\prime}(B_3)}\Big)+\delta_0\|v_h\|_{L^{p^\prime}(B_3)}.
\end{align*}
Finally, we use the exact same covering argument in the proof of Lemma \ref{localStabilityLemma} to obtain
\begin{align*}
\|v_h\|_{L^{p^\prime}(\Omega)}\les \|\Lcalh^* v_h\|_{W^{-2,p^\prime}_h(\Omega)}+\|v_h\|_{W^{-1,p^\prime}(\Omega)}+\delta_0\|v_h\|_{L^{p^\prime}(\Omega)}.
\end{align*}
Taking $\delta_0$ sufficiently small and kicking back the last term then
yields the G\"arding-type estimate
\begin{align}\label{AppStep2Line}
\|v_h\|_{L^{p^\prime}(\Omega)}\les \|\Lcalh^* v_h\|_{W^{-2,p^\prime}_h(\Omega)}+\|v_h\|_{W^{-1,p^\prime}(\Omega)}.
\end{align}

\noindent {\em Step 3:  A duality argument}\\
\indent In the last step of the proof, we shall combine a duality argument and
\eqref{AppStep2Line} to obtain the desired result \eqref{AdjointGoal}.

Define the set
\begin{align*}
X = \{g\in W^{1,p}_0(\Omega):\ \|g\|_{W^{1,p}(\Omega)} = 1\}.
\end{align*}
Since $X$ is precompact in $L^{p}(\Omega)$, and 
due to the elliptic regularity estimate $\|\varphi\|_{W^{2,p}(\Omega)}\les \|\Lcal \varphi\|_{L^p(\Omega)}$,
the set
\begin{align*}
W  = \{\varphi\in W^{2,p}\cap W^{1,p}_0(\Omega):\ \Lcal \varphi = g,\ \exists g\in X\}
\end{align*}
is precompact in $W^{2,p}(\Omega)$.  Therefore by \cite[Lemma 5]{SchatzWang96}, for
every $\eps>0$, there exists a $h_2(\eps,W)>0$ such 
that for each $\varphi\in W$ and $h\le h_2$ there exists $\varphi_h\in V_h$
satisfying 
\begin{align}\label{DensityApprox}
\|\varphi-\varphi_h\|_{W^{2,p}_h(\Omega)}\le \eps\quad \text{for }k\ge 2.
\end{align}
Note that \eqref{DensityApprox} implies $\|\varphi_h\|_{W^{2,p}_h(\Omega)}\le \|\varphi\|_{W^{2,p}(\Omega)}+\eps\les 1$.

For $g\in X$ we shall use $\varphi_g\in W$ to denote
the solution to $\Lcal \varphi_g=g$.  We then have by Lemma \ref{LContLemma}, for any 
$v_h\in V_h$ and $\varphi_h\in V_h$,
\begin{align*}
\int_\Omega v_h g\,dx  
&= ( \Lcalh \varphi_g,v_h) = ( \Lcalh^* v_h,\varphi_h) + ( \Lcalh (\varphi_g-\varphi_h),v_h)\\
&\les \|\Lcalh^* v_h\|_{W_h^{-2,p^\prime}(\Omega)}\|\varphi_h\|_{W^{2,p}_h(\Omega)} + \|\varphi_g-\varphi_h\|_{W^{2,p}_h(\Omega)} \|v_h\|_{L^{p^\prime}(\Omega)}.
\end{align*}
Choosing $\varphi_h$ so that \eqref{DensityApprox} is satisfied (with $\varphi=\varphi_g$) and 
using the definition of the $W^{-1,p^\prime}$ norm \eqref{Hm1norm}
results in
\begin{align*}
\|v_h\|_{W^{-1,p^\prime}(\Omega)} \les \|\Lcalh^* v_h\|_{W^{-2,p^\prime}_h(\Omega)} + \eps \|v_h\|_{L^{p^\prime}(\Omega)}.
\end{align*}
Finally we apply this last estimate in \eqref{AppStep2Line} to obtain
\begin{align*}
\|v_h\|_{L^{p^\prime}(\Omega)}\les \|\Lcalh^* v_h\|_{W^{-2,p^\prime}_h(\Omega)}+\eps\|v_h\|_{L^{p^\prime}(\Omega)}.
\end{align*}
Taking $\eps$ sufficiently small and kicking back a term to the left--hand side
yields \eqref{AdjointGoal}.  This completes the proof.

\section{A discrete Poincar\'e estimate}
\begin{lemma}\label{discretePoincare}
There holds for any $w_h\in V_h(D)$ with ${\rm diam}(D) \ge h$,
\begin{align*}
\|w_h\|_{W^{m,p}(D)}\les {\rm diam}(D)^{2-m}\|w_h\|_{W^{2,p}_h(D)}\quad m=1,2.
\end{align*}
\end{lemma}
\begin{proof}
Denote by $V_{c,h}\subset H^2(\Omega)\cap H^1_0(\Omega)$
the Argyris finite element space \cite{Brenner},
and let $E_h:V_h\to V_{c,h}$ be the enriching
operator constructed in \cite{BrennerSung05} by averaging.
The arguments in \cite{BrennerSung05}
and scaling show that, for $w_h\in V_h(D)$,
\begin{align}\label{EhProp}
E_h w_h\in H^2_0(D_h),\quad \|w_h-E_h w_h\|_{W^{m,p}(D)}\les h^{2-m}\|w_h\|_{W^{2,p}_h(D)}\ (m=0,1,2),
\end{align}
where $D_h$ is given by \eqref{Dhdef}.
Since $E_h w_h\in H^2_0(D_h)$
and ${\rm diam}(D)\ge h$,
the usual Poincar\'e inequality gives
\begin{align*}
\|E_h w_h\|_{W^{m,p}(D_h)}\les {\rm diam}(D_h)^{2-m}\|E_h w_h\|_{W^{2,p}(D_h)}
\les {\rm diam}(D)^{2-m}\|w_h\|_{W^{2,p}_h(D)}.
\end{align*}
Therefore by adding and subtracting terms, we obtain
for $m=0,1$,
\begin{align*}
\|w_h\|_{W^{m,p}(D)}
&\le \|E_h w_h\|_{W^{m,p}(D)}+\|w_h-E_hw_h\|_{W^{m,p}(D)}\\
&\les {\rm diam}(D)^{2-m}\|w_h\|_{W^{2,p}_h(D)}+h^{2-m}\|w_h\|_{W^{2,p}_h(D)}\\
&\les {\rm diam}(D)^{2-m}\|w_h\|_{W^{2,p}_h(D)},
\end{align*}
where again, we have used the assumption $h\le {\rm diam}(D)$.
The proof is complete.\hfill
\end{proof}

\section{A discrete H\"older inequality}
\begin{lemma}\label{discreteHolder}
For any smooth function $\eta$, and $w_h\in V_h(D), v_h\in V_h$, there holds
\begin{align*}
\int_D (\nab \eta \cdot \nab w_h)v_h\, dx \les \|\eta \|_{W^{2,\infty}(D)}\|w_h\|_{W^{2,p}_h(D)}\|v_h\|_{W^{-1,p^\prime}(D)}.
\end{align*}
\begin{proof}
Let $E_h:V_h\to V_c$ be the enriching operator
in Lemma \ref{discretePoincare} satisfying \eqref{EhProp}.
Since $E_h w_h\in H^2(D)$ we have
\begin{align*}
\int_D (\nab \eta \cdot \nab (E_hw_h))v_h\, dx 
&\les 
\|\nab \eta \cdot \nab (E_hw_h)\|_{W^{1,p}(D)}\|v_h\|_{W^{-1,p^\prime}(D)}\\
&\les 
\|\eta\|_{W^{2,\infty}(D)} \|E_hw_h\|_{W^{2,p}(D)}\|v_h\|_{W^{-1,p^\prime}(D)}\\
&\les 
\|\eta\|_{W^{2,\infty}(D)} \|w_h\|_{W^{2,p}_h(D)}\|v_h\|_{W^{-1,p^\prime}(D)}
\end{align*}
Combining this estimate with the triangle inequality, \eqref{EhProp}, and an inverse estimate gives
\begin{align*}
\int_D (\nab \eta \cdot \nab w_h)v_h\, dx 
& = \int_D (\nab \eta \cdot \nab (E_hw_h))v_h\, dx 
+\int_D (\nab \eta \cdot \nab (w_h-E_hw_h))v_h\, dx \\
&\les \|\eta\|_{W^{2,\infty}(D)} \big(\|w_h\|_{W^{2,p}_h(D)}\|v_h\|_{W^{-1,p^\prime}(D)}
+ \|w_h-E_h w_h\|_{W^{1,p}(D)}\|v_h\|_{L^{p^\prime}(D)}\big)\\
&\les \|\eta\|_{W^{2,\infty}(D)} \|w_h\|_{W^{2,p}_h(D)}\|v_h\|_{W^{-1,p^\prime}(D)}.
\end{align*}
\hfill
\end{proof}
\end{lemma}

\end{document}